\documentclass[11pt, oneside, reqno]{amsart}
\usepackage{mathtools}
\usepackage{graphicx}
\usepackage{tikz}
\usepackage[style=alphabetic, maxnames=10]{biblatex}
\usepackage{amsmath} 
\usepackage{amsthm}
\usepackage{amsfonts} 
\usepackage{accents}
\usepackage[a4paper, portrait, margin=1in]{geometry}
\usepackage{amssymb}
\usepackage{wasysym}
\usepackage{comment}
\usepackage{hyperref}

\newtheorem{theorem}{Theorem}[section]
\newtheorem{lemma}[theorem]{Lemma}
\newtheorem{definition}[theorem]{Definition}
\newtheorem{prop}[theorem]{Proposition}
\newtheorem{corollary}[theorem]{Corollary}
\theoremstyle{definition}
\newtheorem*{remark}{Remark}

\numberwithin{equation}{section}

\bibliography{CCoSSiFFVV_Journal_v01.bib}

\hypersetup{
    colorlinks=true,
    linkcolor=blue,
    filecolor=magenta,   
    citecolor = magenta,
    urlcolor=magenta,
    pdftitle={Overleaf Example},
    pdfpagemode=FullScreen,
    }

\newcommand{\eoq}[2]{\frac{|E|^{#1}}{q^{#2}}}

\newcommand{\oq}[1]{O_{#1}({\mathbb{F}_q})}

\newcommand{\ppc}{p}
\newcommand{\oqsize}[1]{| O_{d-#1}(\mathbb{F}_q) | }
\newcommand{\stsum}{\sum_{\delta \in \mathbb{D}} \nu^2_{\TT}(\delta)}
\newcommand{\stsumc}{\sum_{\delta \in \mathbb{D}} \nu^2_{\CC}(\delta)}

\newcommand{\TT}{\mathcal{T}}
\newcommand{\CC}{\mathcal{C}}

\newcommand{\geqsim}{\gtrsim}
\newcommand{\Stab}{\mathrm{Stab}}
\newcommand{\FF}{\mathbb{F}}

\newcommand{\fqd}{\mathbb{F}_q^d}

\makeatletter
\def\subsection{\@startsection{subsection}{2}%
  \z@{.5\linespacing\@plus.7\linespacing}
{.5\baselineskip}%
  {\normalfont\centering\scshape}%
}
\makeatother

\author{Timothy Cheek}
\address{Department of Mathematics, University of Michigan}
\email{\href{mailto:timcheek@umich.edu}{timcheek@umich.edu}}

\author{Joseph Cooper}
\address{Department of Pure Mathematics and Mathematical Statistics, University of Cambridge}
\email{\href{mailto:jc2407@cam.ac.uk}{jc2407@cam.ac.uk}}

\author{Pico Gilman}
\address{Department of Mathematics, University of California Santa Barbara}
\email{\href{mailto:picogilman@ucsb.edu}{picogilman@ucsb.edu}}

\author{Alex Iosevich}
\address{Deparment of Mathematics, University of Rochester}
\email{\href{mailto:alex.iosevich@rochester.edu}{alex.iosevich@rochester.edu}}

\author{Kareem Jaber}
\address{Department of Mathematics, Princeton University}
\email{\href{mailto:kj5388@princeton.edu}{kj5388@princeton.edu}}

\author{Eyvindur Palsson}
\address{Deparment of Mathematics, Virginia Tech}
\email{\href{mailto:palsson@vt.edu}{palsson@vt.edu}}

\author{Vismay Sharan}
\address{Department of Mathematics, Yale University}
\email{\href{mailto:vismay.sharan@yale.edu}{vismay.sharan@yale.edu}}

\author{Jenna Shuffelton}
\address{Department of Mathematics and Statistics, Williams College}
\email{\href{mailto:jms13@williams.edu}{jms13@williams.edu}}

\author{Marie-H\'el\`ene Tom\'e}
\address{Department of Mathematics, Duke University}
\email{\href{mailto:mariehelene.tome@duke.edu}{mariehelene.tome@duke.edu}}

\makeatletter
\renewcommand*\@maketitle{%
  \normalfont\normalsize
  \@adminfootnotes
  \@mkboth{\@nx\shortauthors}{\@nx\shorttitle}%
  \global\topskip42\p@\relax 
  \@settitle
  \ifx\@empty\authors \else \@setauthors \fi
  \ifx\@empty\@date \else {\vskip 1em \vtop{\centering\large\@date\@@par}}\fi
  \ifx\@empty\@dedicatory
  \else
    \baselineskip18\p@
    \vtop{\centering{\footnotesize\itshape\@dedicatory\@@par}%
      \global\dimen@i\prevdepth}\prevdepth\dimen@i
  \fi
  \@setabstract
  \normalsize
  \if@titlepage
    \newpage
  \else
    \dimen@34\p@ \advance\dimen@-\baselineskip
    \vskip\dimen@\relax
  \fi
} 
\renewcommand*\@adminfootnotes{%
  \let\@makefnmark\relax  \let\@thefnmark\relax
  \ifx\@empty\@subjclass\else \@footnotetext{\@setsubjclass}\fi
  \ifx\@empty\@keywords\else \@footnotetext{\@setkeywords}\fi
  \ifx\@empty\thankses\else \@footnotetext{%
    \def\par{\let\par\@par}\@setthanks}%
  \fi
}
\makeatother

\let\oldtocsection=\tocsection

\let\oldtocsubsection=\tocsubsection

\renewcommand{\tocsection}[2]{\hspace{0em}\oldtocsection{#1}{#2}}
\renewcommand{\tocsubsection}[2]{\hspace{2em}\oldtocsubsection{#1}{#2}}
\usepackage{xpatch}
\makeatletter   
\xpatchcmd{\@tocline}
{\hfil\hbox to\@pnumwidth{\@tocpagenum{#7}}\par}
{\ifnum#1<0\hfill\else\dotfill\fi\hbox to\@pnumwidth{\@tocpagenum{#7}}\par}
{}{}
\makeatother

\begin{document}
\date{\today}
\title{Congruence Classes of Simplex Structures in Finite Field Vector Spaces}




\numberwithin{equation}{section}

\setcounter{section}{0}

\begin{abstract}
We study a generalization of the Erd\H{o}s-Falconer distance problem over finite fields. For a graph $G$, two embeddings $p, p': V(G) \to \mathbb{F}_q^d$ of a graph $G$ are congruent if for all edges $(v_i, v_j)$ of $G$ we have that $||p(v_i) - p(v_j)|| = ||p'(v_i) - p'(v_j)||$. What is the infimum of $s$ such that for any subset $E\subset \mathbb{F}_q^d$ with $|E| \gtrsim q^s$, $E$ contains a positive proportion of congruence classes of $G$ in $\mathbb{F}_q^d$? Bennett et al. and McDonald used group action methods to prove results in the case of $k$-simplices. The work of Iosevich, Jardine, and McDonald as well as that of Bright et al. have proved results in the case of trees and trees of simplices, utilizing the inductive nature of these graphs.

Recently, Aksoy, Iosevich, and McDonald combined these two approaches to obtain nontrivial bounds on the ``bowtie" graph, two triangles joined at a vertex. Their proof relies on an application of the Hadamard three-lines theorem to pass to a different graph. We develop novel geometric techniques called branch shifting and simplex unbalancing to reduce our analysis of trees of simplices to a much smaller class of simplex structures. This allows us to establish a framework that handles a wide class of graphs exhibiting a combination of rigid and loose behavior. In $\mathbb{F}_q^2$, this approach gives new nontrivial bounds on chains and trees of simplices. In $\mathbb{F}_q^d$, we improve on the results of Bright et al. in many cases and generalize their work to a wider class of simplex trees. We discuss partial progress on how this framework can be extended to more general simplex structures, such as cycles of simplices and structures of simplices glued together along an edge or a face.
\end{abstract}

\maketitle

\begin{center}
\textit{``The shortest path between two truths in the real domain passes through the complex domain."}

 - Jacques Hadamard
\end{center}

\tableofcontents

\section{Introduction}

The Erd\H{o}s distance problem is a classical problem in geometric combinatorics. Originally introduced by Erd\H{o}s in 1946, it asks for the smallest number of distinct distances determined by a subset of $\mathbb{R}^d$ under the standard Euclidean metric. The original conjecture in \cite{erdos} states that for a subset $P\subset \mathbb{R}^d$, $d\geq 2$, with size $|P|=N$, the least number of distances determined by $P$ is on the order of $N^{\frac{2}{d}-\epsilon}$ for $\epsilon>0$. The case of $d=2$ was resolved by Guth and Katz in 2011 within a $\sqrt{\log N}$ factor of the best known lower bound (\cite{guthkatz}) . The case where $d\geq 3$ is still open with best known bounds from Solymosi and Vu in 2008 (\cite{vusolymosi}). The Falconer distance problem is similar in nature to the original Erd\H{o}s distance problem and asks for the minimum Hausdorff dimension of a set $E\subset \mathbb{R}^d$ needed to achieve a positive Lebesgue measure of distances.  In \cite{ogfalconer}, Falconer conjectured that if the Hausdorff dimension of $E$ is strictly greater than $\frac{d}{2}$, the set of distances $\{||x - y||, \, x,y \in E \}$ with the standard Euclidean metric has positive (one-dimensional) Lebesgue measure.  This conjecture is still open in all dimensions.  For $d=2$, the best threshold is $\frac{5}{4}$ from Guth et al. in 2018 (\cite{falconer2d}), and for $d\geq3$ the best known bound is $\frac{d}{2}+\frac{1}{4}-\frac{1}{8d+4}$ were found by Du et al. in 2023 (\cite{falconerhigherd}). \bigskip

In \cite{iosevichrudnev}, Iosevich and Rudnev introduced the Erd\H{o}s-Falconer distance problem over finite fields, an arithmetic analogue to these two problems.  In $\mathbb{F}_q^d$, for $q$ a prime power, we define the distance between two points $x, y \in \mathbb{F}_q^d$ as
\begin{align}
    ||x - y||\ =\ (x_1 - y_1)^2 + \cdots + (x_d - y_d)^2. \label{dist defn}
\end{align}

The Erd\H{o}s-Falconer distance problem over $\fqd$ asks for the minimum size of a subset $E$ such that pairs of points in $E$ determine all possible distances in $\mathbb{F}_q$, or more generally, a positive proportion of distances in $\mathbb{F}_q$\footnote{Although this is not a metric, it is standard in the literature to refer to it as a distance due to historical connections to the Erdos and Falconer distance problems.}. More precisely, it asks for the infimum of all $s \in \mathbb{R}$ such that $|E| \gtrsim q^s$ implies that $E$ contains all distances, or a positive proportion of distances, in $\mathbb{F}_q$, where $q$ is treated as an asymptotic parameter.

\begin{remark}
    Here, as in most other papers on the subject, we write $X \lesssim Y$ to mean there is a constant $C$ with $X \leq CY$ independent of $q$, and similarly, we write $X \geqsim Y$ if $X \geq CY$ for a constant $C$ independent of $q$. We say $X \approx Y$, or that $X$ is on the order of $Y$, if $X \lesssim Y \lesssim X$.
\end{remark}

This problem has been actively researched over the past two decades.  In \cite{iosevichrudnev}, it was proved that for $E \subset \fqd$ and $s = \frac{d+1}{2}$, if $E \gtrsim q^s$ then $E$ contains all distances in $\mathbb{F}_q$.  In \cite{hyperplanes}, the authors showed that this $s$ is tight in odd dimensions even for the weaker problem of considering a positive proportion of distances.  In $d = 2$, this exponent has been improved when considering a positive proportion of distances, first to $\frac{4}{3}$ in \cite{pinneddistsets}, and recently to $\frac{5}{4}$ in \cite{fivefourths}, which is the best known result as of writing.  \bigskip 


An interesting analogue of the original Erd\H{o}s-Falconer distance problem in $\fqd$
generalizes the notion of the distance set determined by a set of points $E$ in $\mathbb{F}_q^d$.  

\begin{definition}
    Given a graph $G$, we say that two embeddings $p, p': V(G) \to \mathbb{F}_q^d$ of a graph $G$ are congruent if for every $(v_i, v_j)$ in the edge set of $G$, we have that $||p(v_i) - p(v_j)|| = ||p'(v_i) - p'(v_j)||$. 
\end{definition} 

For a general graph $G$, what is the infimum of $s$ such that $|E| \geqsim q^s$ implies that $E$ contains an embedding of $G$ in each congruence class, or a positive proportion of congruence classes of $G$ in $\mathbb{F}_q^d$?  The original Erd\H{o}s-Falconer distance problem over $\fqd$ relates to the case where $G$ is the complete graph on two vertices.  \bigskip

A natural graph to consider in this generalization is the complete graph on $k+1$ vertices, whose embeddings correspond to $k$-simplices in $\mathbb{F}_q^d$.  To this end, we often refer to both the complete graph on $k+1$ vertices and embeddings of this graph in $\mathbb{F}_q^d$ as simplices of dimension $k$.  In \cite{groupactions}, Bennett et al. proved that for $n \leq d$, $E \subset \mathbb{F}_q^d$, and $s = \frac{dn+1}{n+1}$, then $E$ contains a positive proportion of congruence classes of $n$-simplices in $\fqd$ whenever $|E| \gtrsim q^s$.  To achieve this, the authors adapted group action methods from the continuous case of the Falconer distance problem, see \cite{eyvipaper}.  While the distance function in \eqref{dist defn} isn't a norm, it is still invariant under transformations in $\oq{d}$, the group of orthogonal transformations over $\fqd$ with respect to the standard inner product.  These group action methods over finite fields have found success in proving results for rigid graphs.  For $d=2$, \cite{groupactions} improved the $s = \frac{dn+1}{n+1}$ exponent for 2-simplices due to a stronger bound on the finite field analogue of the Mattila integral in this setting.  In 2019, McDonald \cite{alexmcdonald} used group action techniques to extend the results of \cite{groupactions} in both the $\fqd$ and $\mathbb{F}_q^2$ settings to $n$-simplices for $n \geq d$.  \bigskip

Results have been proven for loose graphs as well, such as paths, cycles, and trees, where a group action approach is less clear.  In this setting, success was found in using techniques that take advantage of the inductive nature of these graphs.   In \cite{longpaths}, the authors proved that for $G$ a path of arbitrary length, $E \subset \mathbb{F}_q^d$ and $s = \frac{d+1}{2}$, if $|E| \gtrsim q^{s}$, then $E$ contains at least one embedding of $G$ in every congruence class, and so $E$ contains all congruence classes of $G$ in $\mathbb{F}_q^d$.  Similar results have been proved in \cite{cycles} when $G$ is a tree for $s = \frac{d+1}{2}$.  These results have been generalized in \cite{small2023}, where for simplices of sufficiently small dimension relative to the ambient space, the authors proved that for $E \subset \mathbb{F}_q^d$, and $s = m+\frac{d-1}{2}$, if $|E| \gtrsim q^s$ then $E$ contains all congruence classes of embeddings of a certain class of trees of $m$-simplices. \bigskip

In \cite{edgedeletion}, Iosevich and Parshall proved a much more general result: for any connected graph $G$ with maximum vertex degree $t$, $E \subset \mathbb{F}_q^d$, and $s = t + \frac{d-1}{2}$, if $|E| \gtrsim q^s$ then $E$ contains at least one embedding of $G$ in every congruence class.  While this result gives nontrivial $s$ for any $G$ in a sufficiently high dimension space, the specialized results of \cite{groupactions} and \cite{alexmcdonald} yield better results on high-dimensional simplices, and \cite{small2023} yields better results on the class of simplex trees they handle for all dimensions of the space.  This suggests that the techniques that are used to work with a graph $G$ should vary with the properties that $G$ exhibits. \bigskip


Very recently, in \cite{preprint}, Aksoy, Iosevich, and McDonald took a step toward combining the group action techniques for simplices and the inductive techniques for paths and trees to handle the bowtie graph $B$ shown in Figure \ref{bowtieandkite}, a graph that exhibits a combination of rigid and loose behavior.  To do this, the authors first reduced the problem to bounding a sum corresponding to the number of pairs of congruent embeddings of $B$ in $E$, following a similar setup as in \cite{groupactions}.  Then, by the Hadamard three-lines theorem from complex analysis, they convert the problem from bounding the sum corresponding to $B$ (a 2-simplex attached to a 2-simplex) to a sum corresponding to the kite graph $K$ (a 3-simplex attached to a 1-simplex), as shown in Figure~\ref{bowtieandkite}. \bigskip

\begin{figure}[hbt!]
\centering

\tikzset{every picture/.style={line width=0.75pt}} 

\begin{tikzpicture}[x=0.75pt,y=0.75pt,yscale=-1,xscale=1]

\draw   (170.33,49) -- (104.33,85.33) -- (104.33,12.67) -- cycle ;
\draw   (170.33,49) -- (236.33,12.67) -- (236.33,85.33) -- cycle ;
\draw    (436,82) -- (504.67,82) ;
\draw   (329.25,21) -- (403.98,21) -- (436,82) -- (361.28,82) -- cycle ;
\draw    (403.98,21) -- (361.28,82) ;
\draw    (329.25,21) -- (436,82) ;

\draw (163,98) node [anchor=north west][inner sep=0.75pt]   [align=left] {$\displaystyle B$};
\draw (400,98) node [anchor=north west][inner sep=0.75pt]   [align=left] {$\displaystyle K$};

\end{tikzpicture}

\caption{The ``bowtie'' graph, $B$, and the ``kite'' graph, $K$.}
\label{bowtieandkite}
\end{figure}
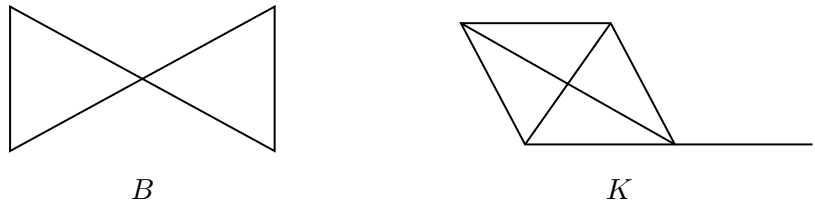

Morally, as paths are much more abundant than high-dimensional simplices, if $E$ is large enough to contain a positive proportion of congruence classes of 3-simplices, then $E$ should also contain a positive proportion of congruence classes of $K$. To prove this rigorously, the authors split up the sum corresponding to $K$ to isolate the contributions from the path and those from the 3-simplex.  Then, the path terms can be bounded using the inductive techniques from \cite{longpaths}, and the 3-simplex terms could be bounded using the group action techniques from \cite{groupactions}. \bigskip

In this paper, we generalize the techniques of \cite{preprint} to obtain nontrivial bounds on $s$ for a much wider and more complex class of graphs exhibiting a combination of rigid and loose structure.  In particular, we focus our attention on a class of graphs that we call simplex trees.

\begin{definition}
For a simplex $S$, let $V(S)$ denote the set of vertices of $S$. Let $\TT$ be a single simplex or let $\TT = \{S_i : i \in I\}$ be a finite set of simplices. We denote by $\mathcal{N}_\TT(S_i)$ the set of all simplices in $\TT$ sharing a vertex with $S_i$. We drop the subscript $\TT$ when it is clear from context.  We say $\TT$ is a \textbf{(connected) simplex tree} if $|\TT| = 1$ or if $|\TT| > 1$ and $\TT$ satisfies the following properties:
\begin{enumerate}
    \item For any $i, j \in I$ distinct, we have that $|V(S_i) \cap V(S_j)| \leq 1$.
    \item For each $a, b \in I$, there exists a path of simplices in $\TT$ from $S_a$ to $S_b$, i.e., there exist distinct integers $i_1,\ldots,i_n \in I$ such that $S_{i_1} = S_a$, $S_{i_n} = S_b$ and $S_{i_{j+1}} \in \mathcal{N}(S_{i_{j}})$ for each $1 \leq j \leq n-1$.
    \item There is no closed loop of simplices, i.e. there do not exist distinct integers $i_1, \ldots, i_n \in I$ such that for each $1 \leq j \leq n$, we have that $S_{i_{j+1}} \in \mathcal{N}(S_{i_j})$ and $S_{i_n} \in \mathcal{N}(S_{i_1})$.
Note that for a connected simplex tree, $S_i \in \mathcal{N}(S_j)$ if and only if $|V(S_i) \cap V(S_j)| = 1$. 
\end{enumerate}

\end{definition}

Our definition may be viewed as a natural generalization of an undirected tree, where each simplex may be viewed as a vertex and each vertex shared by two simplices may be viewed as an edge between them. 

\begin{remark}
    Our definition of a simplex tree is more general than the definition given in \cite{small2023}. Their class of simplex trees is formed by taking a standard graph-theoretic tree, and to each edge, attaching a $k$-simplex sharing this edge.  This means that for any simplex in a tree considered by \cite{small2023}, only two of its vertices may be shared by other simplices.  Under our definition of simplex trees, however, a simplex can share any number of its vertices with other simplices. Secondly, while the definition in \cite{small2023} requires all simplices in the tree to have the same dimension, in our definition, simplices can have different dimensions throughout the tree.
\end{remark}

\subsection{Statement of main results}

Our main theorem deals with how big $E$ needs to be to ensure it contains a positive proportion of congruence classes of $\TT$.

\begin{theorem} \label{maintheoremd}
    For a simplex tree $\TT$ and an integer $k$ such that $1 \leq k < \frac{d+1}{2}$, define 
    \begin{equation}
        N_k\ =\ k+\sum_{\substack{S \in \TT \\ \dim(S) > k}} (\dim(S) -k).
    \end{equation}
    For an $E \subset \FF_q^d$, suppose that for $s =\max \left(\frac{dN_k + 1}{N_k+1}, \, \, k + \frac{d-1}{2} \right) $, $|E|\ \gtrsim\ q^{s}$. Then $E$ contains a positive proportion of congruence classes of embeddings of $\TT$ in $\FF_q^d$.
\end{theorem}
Note that for a given $\TT$ and $k$, Theorem \ref{maintheoremd} results in some $s < d$, and therefore this result is nontrivial for all simplex trees $\TT$. Naively, for any graph $G$ and a space $\mathbb{F}_q^d$, one can find a nontrivial $s$ (i.e. $s < d$) such that for any $E \subset \mathbb{F}_q^d$ and $|E| \gtrsim q^s $, $E$ contains a positive proportion of congruence classes of $G$ in $\fqd$ by first finding an $s$ for the complete graph on the number of vertices of $G$, and then applying the corresponding results of \cite{groupactions} and \cite{alexmcdonald}. However, Theorem \ref{maintheoremd} and all subsequent theorems in the introduction yield better bounds than this method.
\bigskip

The idea of the proof of Theorem \ref{maintheoremd} is as follows.  We first generalize the setup of \cite{groupactions} to reduce the problem to bounding a sum corresponding to the number of pairs of congruent embeddings of $\TT$ in $E$.  We develop two techniques called \textit{branch shifting} and \textit{simplex unbalancing} that algebraically modify this sum in ways that correspond to geometric modifications of $\TT$.  This allows us to reduce to the problem to bounding a sum corresponding to the modified $\TT$.  We apply these techniques repeatedly until our modified $\TT$ consists of one large simplex of dimension $N_k$ with trees of $k$-simplices attached to a single vertex.  From here, we use Fourier analytic techniques to isolate the contributions of each of these two different structures from the sum, handling the large simplex with group action techniques in the spirit of \cite{groupactions} and \cite{alexmcdonald} and handling the smaller trees of simplices with the inductive techniques from \cite{cycles} and \cite{small2023}.  Once $E$ is large enough to contain a positive proportion of congruence classes of $N_k$-simplices and all congruence classes of trees of $k$-simplices, which occur at $s = \frac{dN_k + 1}{N_k + 1}$ and $s = k + \frac{d-1}{2}$, respectively, we show that $E$ contains a positive proportion of congruence classes of $\TT$. Therefore, the geometry of our reduced structure dictates our threshold for $s$. \bigskip

The value of $k$ in Theorem \ref{maintheoremd} can be viewed as a parameter that determines how much the resulting $s$ relies on inductive techniques rather than group action techniques. When $k$ is small, we modify $\TT$ to have a very large central simplex and simplex trees of very small dimension attached to a vertex. Then, the strength of our result relies much more heavily on group action techniques rather than inductive techniques. On the other hand, when $k$ is large, we modify $\TT$ to have a smaller central simplex, while the simplex trees sticking out of it are of higher dimension. In this scenario, the group action techniques yield a smaller $s$, at the cost of inductive techniques yielding a larger $s$ to guarantee that $E$ contains all congruence classes of trees of $k$-simplices. The flexibility we gain by adaptively scaling between group action techniques and inductive techniques yields a vast generalization of the methods of \cite{preprint}. See Figure \ref{kscalingfig} for an illustration of how scaling in $k$ affects the modification of $\TT$. Note that the edges inside of the simplices are omitted for clarity. \bigskip

\begin{figure}[hbt!]
\centering

\tikzset{every picture/.style={line width=1pt}} 

\begin{tikzpicture}[x=0.75pt,y=0.75pt,yscale=-.75,xscale=.75]

\draw   (171.8,113.81) -- (196.13,155.63) -- (147.47,155.63) -- cycle ;
\draw    (196.13,155.63) -- (205.59,191.13) ;
\draw   (100.65,166.22) -- (147.47,155.63) -- (158.07,202.45) -- (111.25,213.04) -- cycle ;
\draw   (171.8,113.81) -- (147.47,72) -- (196.13,72) -- cycle ;
\draw   (205.59,191.13) -- (248.67,212.32) -- (227.48,255.39) -- (184.41,234.2) -- cycle ;
\draw  [color={rgb, 255:red, 245; green, 166; blue, 35 }  ,draw opacity=1 ] (453.05,77.95) -- (439.57,110.48) -- (407.05,123.95) -- (374.52,110.48) -- (361.05,77.95) -- (374.52,45.43) -- (407.05,31.95) -- (439.57,45.43) -- cycle ;
\draw [color={rgb, 255:red, 189; green, 16; blue, 224 }  ,draw opacity=1 ]   (439.57,45.43) -- (479.67,30) ;
\draw [color={rgb, 255:red, 189; green, 16; blue, 224 }  ,draw opacity=1 ]   (479.67,30) -- (516.67,45) ;
\draw [color={rgb, 255:red, 189; green, 16; blue, 224 }  ,draw opacity=1 ]   (516.67,45) -- (534.67,88) ;
\draw [color={rgb, 255:red, 189; green, 16; blue, 224 }  ,draw opacity=1 ]   (516.67,45) -- (564.67,37) ;
\draw  [color={rgb, 255:red, 245; green, 166; blue, 35 }  ,draw opacity=1 ] (476.38,255.92) -- (454.04,294.62) -- (409.35,294.62) -- (387,255.92) -- (409.35,217.21) -- (454.04,217.21) -- cycle ;
\draw  [color={rgb, 255:red, 189; green, 16; blue, 224 }  ,draw opacity=1 ] (474.95,181.28) -- (495.85,217.21) -- (454.04,217.21) -- cycle ;
\draw  [color={rgb, 255:red, 189; green, 16; blue, 224 }  ,draw opacity=1 ] (474.95,181.28) -- (454.04,145.34) -- (495.85,145.34) -- cycle ;
\draw  [color={rgb, 255:red, 189; green, 16; blue, 224 }  ,draw opacity=1 ] (516.76,253.15) -- (495.85,217.21) -- (537.67,217.21) -- cycle ;
\draw  [color={rgb, 255:red, 189; green, 16; blue, 224 }  ,draw opacity=1 ] (516.76,253.15) -- (537.67,289.09) -- (495.85,289.09) -- cycle ;
\draw    (250.67,129) -- (330.87,89.88) ;
\draw [shift={(332.67,89)}, rotate = 154] [color={rgb, 255:red, 0; green, 0; blue, 0 }  ][line width=0.75]    (10.93,-3.29) .. controls (6.95,-1.4) and (3.31,-0.3) .. (0,0) .. controls (3.31,0.3) and (6.95,1.4) .. (10.93,3.29)   ;
\draw    (272,208) -- (354.77,236.35) ;
\draw [shift={(356.67,237)}, rotate = 198.91] [color={rgb, 255:red, 0; green, 0; blue, 0 }  ][line width=0.75]    (10.93,-3.29) .. controls (6.95,-1.4) and (3.31,-0.3) .. (0,0) .. controls (3.31,0.3) and (6.95,1.4) .. (10.93,3.29)   ;

\draw (260.66,102.96) node [anchor=north west][inner sep=0.75pt]  [rotate=-333.32]  {$k=1$};
\draw (298.87,196.08) node [anchor=north west][inner sep=0.75pt]  [rotate=-20.4]  {$k=2$};

\end{tikzpicture}
\caption{Two modifications of a simplex tree corresponding to different $k$.}
\label{kscalingfig}
\end{figure}
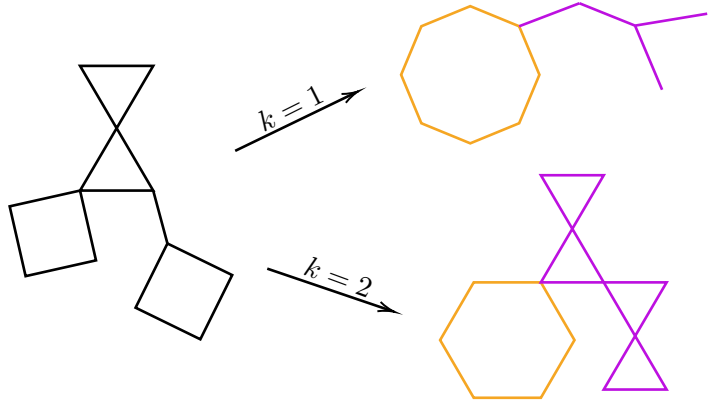

Note that when the dimension of $\fqd$ is much larger than the dimension of simplices in a simplex tree $\TT$, it is always optimal to pick $k$ as the dimension of the largest simplex in $\TT$. However, when the dimension of $\fqd$ is comparable to the dimension of the simplices, the choice of $k$ is more nuanced.  For example, Theorem \ref{maintheoremd} gives the best results for a chain of three triangles in $\mathbb{F}_q^4$ for $k=1$, and gives the best results for a 6-simplex attached to a 3-simplex in $\mathbb{F}_q^6$ for $k=2$. 
 If the dimension of the simplices in $\TT$ is bounded, it becomes optimal to pick a larger $k$ as the number of simplices in $\TT$ increases. \bigskip

In the case where the simplices in our simplex tree are sufficiently small, we get a result using purely inductive techniques.

\begin{theorem}\label{extendsmall2023}
    Consider a simplex tree $\TT$ consisting only of simplices of dimension at most $n < \frac{d+1}{2}$. Fix $E \subset \fqd$ and put $s = n + \frac{d-1}{2}$. If $|E| \gtrsim q^s$ then $E$ contains a positive proportion of congruence classes of $\TT$ in $\fqd$.
\end{theorem}

This result follows by using branch shifting to convert our simplex tree $\TT$ to the form of a tree of $n$-simplices in the sense of \cite{small2023}. From here, using Theorem 1.12 of \cite{small2023} gives the desired result. \bigskip  

When working with simplex trees of only $n$-simplices for $n < \frac{d+1}{2}$, Theorem \ref{extendsmall2023} generalizes the results of \cite{small2023}. We achieve the same exponent $s = n + \frac{d-1}{2}$ in Theorem 1.12 of \cite{small2023} but for a wider class of simplex trees than they consider. In many cases where a simplex tree contains simplices of dimension close to $\frac{d+1}{2}$, Theorem \ref{maintheoremd} improves on the results of \cite{small2023}. For instance, for a chain of three triangles in $\mathbb{F}_q^4$, choosing $k=1$ in Theorem \ref{maintheoremd} gives $s = \frac{17}{5}$, while Theorem 1.9 of \cite{small2023} gives $s= \frac{7}{2}$. It is worth noting that for $n < \frac{d+1}{2}$, \cite{small2023} and Theorem \ref{extendsmall2023} give values of $s$ independent of the size of the tree, but Theorem \ref{maintheoremd} weakens as the size of the tree grows. This is because upon Hadamard three-lines reduction, the results of group action techniques are sensitive to the size of the large simplex, while the results of inductive techniques are independent of the size of the simplex tree. As $n$ grows, Theorem \ref{maintheoremd} becomes more and more reliant on group action techniques, even when $k$ is chosen to be maximal. \bigskip

For all simplex trees, we also improve the results of Iosevich and Parshall in \cite{edgedeletion}. When $\TT$ has a simplex of dimension at least $\frac{d+1}{2}$, the methods in \cite{edgedeletion} give trivial results, and when $\TT$ has all simplices of dimension at most $n$ for $n < \frac{d+1}{2}$, applying Theorem \ref{extendsmall2023} improves on the result of \cite{edgedeletion}. Furthermore, our results are independent of the maximum vertex degree of $\TT$, as Theorems \ref{maintheoremd} and \ref{extendsmall2023} only depend on the set of simplices in $\TT$ rather than their configuration. For example, these theorems yield the same bounds for a chain of $n$-simplices of length $\ell$ and a collection of $\ell$ $n$-simplices sharing a common vertex. This gives further indication, as \cite{small2023} discusses, that maximum vertex degree is not always the right notion of complexity.  \bigskip

Due to a stronger bound on the finite field analogue of the Mattila integral in $\mathbb{F}_q^2$, in this setting, we obtain a slightly better result.

\begin{theorem} \label{maintheorem2}
    For a simplex tree $\TT$, define 
    $$N\ =\ 1+\sum_{\substack{S \in \TT, \\ \dim(S) > 1}} (\dim(S) - 1).$$  
    If $E \subset \FF_q^2$ with $q \equiv 3 \pmod 4$ is such that
    $$|E|\ \gtrsim\ q^{\frac{4N}{2N+1}},$$
    then $E$ contains a positive proportion of congruence classes of embeddings of $\TT$ in $\FF_q^2$.
\end{theorem}

This result is novel in the sense that \cite{edgedeletion} and \cite{small2023} both give trivial results for simplex trees in $\mathbb{F}_q^2$.  The restriction of $q \equiv 3 \pmod 4$ is to rule out the existence of a square root of $-1$ and the resulting radius zero spheres in $\mathbb{F}_q^2$. These cause technical obstructions in obtaining the improved bound on the finite field analogue of the Mattila integral. \bigskip

It is interesting to note that we obtain the same results as in Theorems \ref{maintheoremd} and \ref{maintheorem2} when we delete a number of edges from some complete graphs in our simplex tree. More precisely, in Section 2, we note that for $k \geq d$, a congruence class of an embedding of a complete graph on $k+1$ vertices (i.e. a $k$-simplex) is defined by $d(k+1) - \binom{d+1}{2}$ pairwise distances rather than all $\binom{d+1}{2}$ distances.  Therefore, given a simplex structure, we may remove up to $2\binom{d+1}{2} - d(k+1)$ edges of any simplex of dimension at least $d$ and obtain the same bounds as in Theorems \ref{maintheoremd} and \ref{maintheorem2}. \bigskip



In Section 2, we discuss some preliminaries from Fourier analysis and the techniques used in \cite{groupactions} and \cite{preprint}.  In Section 3, we generalize the methods used in \cite{preprint} to isolate terms corresponding to high-dimensional simplices and low-dimensional simplex trees in the sum corresponding to a simplified simplex tree $\TT$.  In Section 4, we present the proof of Theorems \ref{maintheoremd} and \ref{maintheorem2}.  Given a simplex tree $\TT$, we recursively build a sum corresponding to $\TT$ and reduce the problem to bounding this sum.  Then, we define and prove branch shifting and simplex unbalancing. Finally, we inductively apply these two geometric operations to reduce to a case where we may apply the methods in Section 3.  \bigskip

In Section 5, we discuss partial progress in applying the Hadamard three-lines framework to cycles of simplices. As in Section 4, we are able to reduce to a class of cycles with only one large simplex. However, in bounding the sum corresponding to this simple class of simplex cycles, technical issues arise in trying to separate the terms corresponding to the large simplex and to the small simplices. 
 
 In Section 6, we briefly discuss future work on structures of simplices glued together at an edge or a face, as well as a probabilistic approach to the Erd\H{o}s-Falconer distance problem that may circumvent the technical obstructions in Section 5.







\section{Preliminaries}

\subsection{Analytic preliminaries}

We first state some of the Fourier analytic machinery we need.

\begin{definition} \label{fourierfacts}
Suppose $f\colon\mathbb{F}_q^d \to \mathbb{C}$ and suppose $\chi$ is a nontrivial additive character on $\mathbb{F}_q$. Then we define the $m$-th Fourier coefficient of $f$ as
\begin{equation}
    \widehat{f}(m)\ =\ q^{-d}\sum_{x\in \mathbb{F}_q^d} \chi(-m \cdot x)f(x).
\end{equation}
We also have the usual inverse finite Fourier transform, defined by
\begin{equation}
    f(x)\ =\ \sum_{m\in \mathbb{F}_q^d} \chi(m \cdot x)\widehat{f}(x).
\end{equation}
\end{definition}

\begin{lemma}[Parseval] \label{parsevals}
For $f$ as above, 
\begin{equation}
    \sum_{m\in \fqd}|\widehat{f}(m)|^2\ =\ q^{-d}\sum_{x\in \fqd} |f(x)|^2.
\end{equation}
\end{lemma}

For further reading on the finite Fourier transform see, for example, Chapter 7 of \cite{steinshakarchi}. \bigskip

We now review some tools we use in bounding sums later on in this paper.

\begin{lemma}[Log-convexity of $L^p$-norms] \label{rieszthorin}

For $p_0, p_1 \in \mathbb{R}^+$, define 
\begin{align}
    p_{r}\ =\ \left(\frac{1-r}{p_0} + \frac{r}{p_1}\right)^{-1}.
\end{align}
Then, for $f \in L^{p_0} \cap L^{p_1}$, we have
\begin{align}
    ||f||_{p_r}\ \leq\ ||f||_{p_0}^{1-r} ||f||_{p_1}^r. 
\end{align}
\end{lemma}

\begin{theorem}[Hadamard three-lines theorem]\label{hadamardthreelines} 

Suppose $f(z) = f(x+iy)$ is bounded, continuous on the strip $\{z : \text{Re}(z) \in [a,b]\}$, and holomorphic in the interior of the strip. Put 
\begin{align}
    M(x)\ =\  \sup_{y \in \mathbb{R}} |f(x+iy)|.
\end{align} Then $\log(M(x))$ is convex on $[a,b]$.  Therefore, for any $t \in [0,1]$, we have
\begin{align}
    M(ta+(1-t)b)\ \leq\ M(a)^tM(b)^{1-t}.
\end{align}
\end{theorem}

Throughout this paper, we often apply the Hadamard three-lines Theorem in the setting below, so we state it  here for convenience.

\begin{corollary}\label{repeatedhadamard}
    Consider a function $\psi: \mathbb{R}^2 \to \mathbb{R}$ that can be extended to a function from $\mathbb{C}^2$ to $\mathbb{C}$. 
    Suppose that the complex function $\Psi(z) = \psi(a_1 + z, a_2 - z)$ satisfies the criteria of Lemma \ref{hadamardthreelines} in the strip $\{z \colon \text{Re}(z) \in [-a_1,a_2]\}$, and for each $x \in [-a_1, a_2]$, $|\Psi(x + iy)|$ is maximized at $y=0$. Then, we have that
    \begin{equation}
        \psi(a_1, a_2)\ \leq\ \psi
        (a_1 + a_2, 0)^{s_1} \psi(0, a_1 + a_2)^{s_2},
    \end{equation}
    for some $s_1, s_2 \in \mathbb{R}_{> 0}$ that sum to $1$. \bigskip
\end{corollary}

Finally, we have a finite field variant of Taylor's Theorem that is a special case of Lemma 2.1 of \cite{groupactions}.

\begin{lemma} \label{aimlemma39}
    For any function $\phi: \mathbb{F}_q^d \to \mathbb{R}^{\geq 0}$ and any $n \geq 2$, we have that
    \begin{equation} \sum_{x \in \mathbb{F}_q^d}\phi^n(x)\ \lesssim\ q^{-d(n-1)} ||\phi||_{1}^n + ||\phi||_{\infty}^{n-2} \sum_{x \in \mathbb{F}_q^d}\left( \phi(x) - \frac{||\phi||_{1}}{q^d} \right)^2. \end{equation}
\end{lemma}

\subsection{Geometric preliminaries}

We define the geometric counting function $\lambda_{\theta}$ and highlight some important properties, as in \cite{groupactions}.

\begin{definition} \label{lambdathetadef}
    Given a subset $E \subset \mathbb{F}_q^d$ and a $\theta \in \oq{d}$, the group of orthogonal transformations over $\mathbb{F}_q^d$, define the function $\lambda_{\theta}: \mathbb{F}_q^d \to \mathbb{R}$ by 
    \begin{equation}
        \lambda_\theta(w)\ =\ \big|\{(u, u') \in E^2: u - \theta u' = w\}\big|.\label{lambda theta defn}
    \end{equation}
    We also have $||\lambda_{\theta}(w)||_{1} = |E|^2$, $||\lambda_{\theta}(w)||_{\infty} = |E|$, and $\widehat{\lambda}_{\theta}(0) = \eoq{2}{d}$.
\end{definition}

For all $d \geq 2$, we have the following bounds on sums over $\lambda_{\theta}$.

\begin{prop}[Lemmas 4 and 5 of \cite{alexmcdonald}] \label{lambdathetafactsd}
For $\lambda_\theta$ as defined in $\mathbb{F}_q^d$, we have that for $|E| \gtrsim q^{\frac{dn - d +1}{n}}$, 
\begin{equation}
    \sum_{\theta \in \oq{d}}\sum_{w \in \fqd}  \lambda^n_{\theta}(w)\ \lesssim\ |E|^{2n} q^{-dn + {\binom{d+1}{2}}}.
\end{equation}
We also have, independent of the size of $E$, that 
\begin{equation}
    \sum_{\theta \in \oq{d}}\sum_{w \in \fqd} \left(\lambda_{\theta}(w) - \widehat{\lambda}_{\theta}(0)\right)^2\ \lesssim\ q^{\binom{d}{2}+1}|E|^2.
\end{equation}
\end{prop}

In $\mathbb{F}_q^2$, we have the following result.

\begin{lemma} \label{preprintstuff}
For any function $h:\mathbb{F}_q^2 \to \mathbb{R}_{\geq 0}$, for $q \equiv 3\pmod 4$

\begin{equation}
    ||\widehat{hS_t}||_{4}\ \lesssim\ q^{-\frac{3}{2}}||hS_t||_{2}.
\end{equation}
\end{lemma}

See Theorem 4.4 of \cite{pinneddistsets} for the original proof, or the proof of Lemma 2.4 in \cite{preprint} which more closely matches our notation.  Lemma \ref{preprintstuff} allows us to get better bound on sums of $\lambda_{\theta}$ as presented below.

\begin{prop} \label{lambdathetafacts2}
For $q \equiv 3 \pmod 4$, $\lambda_\theta$ defined as in \eqref{lambda theta defn}, and $|E| \gtrsim q^{\frac{4n-4}{2n-1}}$, we have
\begin{align}
    \sum_{\theta \in \oq{2}}\sum_{w \in \mathbb{F}_q^2}\lambda^n_\theta(w)\ \lesssim \ \eoq{2n}{2n-3}.
\end{align}
Furthermore, independent of the size of $E$, we have
\begin{equation}
    \sum_{\theta \in \oq{2}}\sum_{w \in \mathbb{F}_q^2}(\lambda_\theta-\widehat{\lambda}_\theta(0))^2 \ \lesssim \ |E|^{\frac{5}{2}}q.
\end{equation}
\end{prop}

\begin{proof}
    The second part of the above Proposition is proven in \cite{groupactions}, but more explicitly proven in Lemma 2.4 of \cite{preprint}. To prove the first part of the above Proposition, apply Lemma \ref{aimlemma39} to the sum over $w$, and apply Lemma 2.4 of \cite{preprint}. \bigskip
\end{proof}

We now review some definitions and notation about congruence classes.

\begin{definition}\label{congruenceclasstalk}
    For a graph $G$ with $e(G)$ edges, the \textbf{congruence class} of $G$ corresponding to the values $t_1, \ldots, t_{e(G)} \in \mathbb{F}_q$ is the set of embeddings $\mathcal{E}$ in $G$ such that for each embedding $h\in \mathcal{E}$ and for all $i$, the $i$th edge $(v_{i, 1}, v_{i, 2})$ of $G$ satisfies $||h(v_{i, 1}) - h(v_{i, 2})|| = t_i$.  This forms an equivalence relation on the set of embeddings of $G$.
\end{definition}

When $G$ is a simplex, we can rephrase this definition of congruence classes in terms of group actions. For a $k$-simplex $S = (x_0, \ldots, x_k) \in (\mathbb{F}_q^d)^{k+1}$, the \textbf{rigid-motion congruence class} of $G$ is the set of $k$-simplices $S = (y_0, \ldots, y_k) \in (\mathbb{F}_q^d)^{k+1}$ such that there exists a $\theta \in \oq{d}$ and $w \in \fqd$ such that $x_i = \theta y_i + w$ for all $i$. We now show that these two definitions coincide. \bigskip

 
 A $k$-simplex (more precisely, an embedding of a $k$-simplex) $(x_0, \ldots, x_k) \in (\mathbb{F}_q^d)^{k+1}$ is called \textbf{nondegenerate} if the vectors $x_1 - x_0, \ldots, x_n - x_0$ are linearly independent or they span $\mathbb{F}_q^d$. 
 It is well known that for $k \leq d$, two nondegenerate $k$-simplices $(x_0, \ldots, x_k)$ and $(y_0, \ldots, y_k)$ in $\mathbb{F}_q^d$ are congruent in the sense of Definition \ref{congruenceclasstalk} if and only if there exists a $\theta \in \oq{d}$ and $w \in \mathbb{F}_q^d$ for which $x_i = \theta y_i + w$ for all $i$ (see Lemma 2.11 of \cite{pinneddistsets}). In the case that $k \geq d$, we can form this nondegenerate $k$-simplex as the union of nondegenerate $d$-simplices sharing a common $d-1$-simplex, i.e,
\begin{equation}
    (x_0, \ldots, x_k) \ = \ \bigcup_{i \geq d} (x_0, \ldots, x_{d-1}, x_i)
\end{equation}
Then for two congruent $k$-simplices $(x_0, \ldots, x_k)$ and $(y_0, \ldots, y_k)$ in the sense of Definition \ref{congruenceclasstalk}, for $d \leq i \leq k$ there exists $\theta_i \in \oq{d}, w_i \in \mathbb{F}_q^d$ sending $(x_0, \ldots, x_{d-1}, x_i)$ to $(y_0, \ldots, y_{d-1}, y_i)$.  These $\theta_i$ must be chosen from the two $\theta$ sending $(x_0, \ldots, x_{d-1})$ to $(y_0, \dots, y_{d-1})$ (there are two since this $\theta$ is unique up to reflection).  Therefore, for $k \geq d$, every congruence class of $k$-simplices contains most $2^{d-k+1}$ rigid-motion congruence classes of $k$-simplices defined by rotations and translations, and every rigid-motion congruence class lies in a single congruence class.  We see that rigid-motion congruence classes of simplices are a finer equivalence relation than standard congruence classes, but only by a constant factor.  As our results only care about $E$ containing a positive proportion of congruence classes of embeddings of some graph, we may use these definitions interchangeably for congruence classes of nondegenerate simplices.

\begin{remark}
    In Theorem 7.5 of \cite{groupactions} for $k \leq d$ and Theorem 1 of \cite{alexmcdonald} for $k \geq d$, it is proved that the number of rigid-motion congruence classes of degenerate $k$-simplices is negligible compared to the number of rigid-motion congruence classes of all $k$-simplices.  Therefore, for the rest of the paper, we may assume that all embeddings of simplices are nondegenerate.
\end{remark}

We recall the following fact about embeddings of simplices in \cite{alexmcdonald}.

\begin{lemma} \label{numberofconclassnsimp}
    The number of congruence classes of $n$-simplices in $\mathbb{F}_q^d$ is on the order of
    \begin{align} \begin{cases}
        q^{\binom{n + 1}{2}} & n \leq d \\ 
        q^{d(n+1) - \binom{d + 1}{2}} & n \geq d.
    \end{cases} \end{align}
    We note that these two values agree for $n = d-1$ and $n=d$.
\end{lemma}

We establish the following notation.

\begin{definition}\label{defofc}
    For a simplex $S$, let $c_d(S)$ be the integer such that the number of congruence classes of $S$ in $\mathbb{F}_q^d$ is on the order of $q^{c_d(S)}$. We write $c(S)$ when the dimension is clear.
\end{definition}

Finally, we recall a well-known fact about the size of the orthogonal group over $\fqd$.

\begin{lemma} \label{sizeoforthogonalgroup}
    Let $O_d(\mathbb{F}_q)$ denote the orthogonal group over $\fqd$. 
 Then we have that \begin{equation}
        |O_d(\mathbb{F}_q)|\ \approx\ q^{{\binom{d}{2}}}.
    \end{equation}

\end{lemma}

\subsection{Simplex trees}
We develop notation for simplex trees in order to work with them more easily throughout the paper.  

\begin{definition}
    We define a \textbf{rooted simplex tree $(\TT, S_0)$} to be a connected simplex tree with a designated root simplex $S_0$.
\end{definition}
\begin{definition}
    For a rooted simplex tree $(\TT, S_0)$, we define the \textbf{depth of a simplex $S$} in $\TT$ to be the minimal length $n$ of a path of simplices in $\TT$ from $S_0$ to $S$. Specifically, for $i_0, \ldots, i_m \in I$ such that $S_0 = S_{i_0}$ and $S = S_{i_m}$, we define
    \begin{align}
        n\ \coloneqq\ \min \left\{m \in \mathbb{N} : \exists \: S_{i_1}, \ldots ,S_{i_{m-1}} \in \TT, S_{i_{j+1}} \in \mathcal{N}(S_{i_{j}}) \text{ for all } 0 \leq j \leq m-1 \right\}.
    \end{align}   
\end{definition}

\begin{definition}

    For a rooted simplex tree $(\TT, S_0)$, and a simplex $S \in \TT$, we call a vertex $v$ of $S$ a \textbf{parent vertex of $S$} if it is shared by a simplex of less depth. 
 Note that each non-root simplex has exactly one parent vertex, while the root has no parent vertex.  If $v$ is shared by a simplex of higher depth, it is called a \textbf{child vertex of $S$}.  If $v$ is not shared by any simplex other than $S$, it is called a \textbf{free vertex}.
\end{definition}

A \textbf{free rooted simplex tree $(\TT, S_0, v_0)$} is a rooted simplex tree with designated free vertex $v_0$ of $S_0$. 

\begin{definition}
    For a simplex $S$ in a rooted tree $(\TT, S_0)$, denote $\mathcal{P}_{\TT}(S)$ ($\mathcal{P}$ for ``pruned'') as the free rooted simplex tree consisting of $S$ and all simplices connected to $S$ through a path of simplices in $\TT$ of greater depth than $S$. We define $\mathcal{P}_{\TT}(S)$ to be rooted at $S$ with the designated vertex being the parent vertex of $S$ in $\TT$. We drop the subscript $\TT$ when it is clear from context.
\end{definition}

\begin{definition}
    For each $v \in V(S)$ and simplices $S_{i_1}, ..., S_{i_n} \in \mathcal{N}(S)$ of greater depth sharing the vertex $v$ with $S$, we define the $j$-th \textbf{branch} of $S$ at $v$ to be the free rooted simplex tree $\mathcal{P}(S_{i_j})$ with root $S_{i_j}$ and designated free vertex $v$.  Let $\mathcal{B}_{\TT}(S, v)$ be the set of branches of $S$ at $v$, where the $\TT$ subscript is dropped when it is clear from context. If a simplex $S$ has no branches, we call the simplex a \textbf{leaf}. If the simplex $S$ is dimension greater than $n$ and has branches consisting only of simplices of dimension at most $n$, we call the simplex an \textbf{$n$-leaf}.

\end{definition}

It is clear that any simplex tree with a simplex of dimension greater than $n$ contains an $n$-leaf. \bigskip

Note that choosing a congruence class of each simplex $S \in \TT$ determines a unique congruence class of $\TT$, since it determines the distance between every pair of vertices corresponding to an edge in an embedding of $\TT$.  Furthermore, every congruence class of $\TT$ uniquely determines the congruence class of each $S \in \TT$. This immediately gives us following lemma:
\begin{lemma}\label{numembeddingsofsimplextrees}
    Given a simplex tree $\TT$ we have that the number of congruence classes of $\TT$ in $\mathbb{F}_q^d$ is equal to $q^{\sum_{S \in \TT}c(S)}$.  We denote the quantity $\sum_{S \in \TT}c(S)$ as $c(\TT)$.
\end{lemma}


\begin{remark}
    By the discussion and remark under \ref{congruenceclasstalk}, we may assume that in any embedding of $\TT$, all embeddings of simplices in $\TT$ are nondegenerate.
 
\end{remark}



\section{Functional sums and useful bounds}

In this section, we prove several bounds on sums and products of geometric counting functions that will become relevant in the upcoming sections.  Some of the bounds in this section rely on the results of \cite{small2023}, who work with a weaker definition of simplex trees than how we have defined them.  We recall their definition here:
\begin{definition}[Definition 1.11 of \cite{small2023}]
    We define an \textbf{$k$-weak simplex tree} $T_k$ as the graph obtained from a standard graph-theoretic tree, where each edge is replaced with a $k$-simplex containing this edge.  We assume $k \leq \frac{d+1}{2}$.
\end{definition}

For a $k$-weak simplex tree $T_k$ with a distinguished root vertex $r$, we enumerate each of the $e(T_k) = \ell \binom{k+1}{2}$ edges in $T_k$, where $\ell$ is the number of simplices in $T_k$.  For $t_1, \ldots, t_{e(T_k)} \in \mathbb{F}_q$ such that $t_i \neq 0$, we identify the congruence class of embeddings of $T_k$ in $\mathbb{F}_q^d$ by assigning the distance $t_i$ to the $i$th edge.  For an $E \subset \mathbb{F}_q^d$, we define the function $f_{T_k, t_1, \ldots, t_{e(T_k)}}(x)$ as the number of embeddings $h: V(T_k) \to E$ of $T_k$ in the congruence class defined by $t_1, \ldots, t_{e(T_k)} \in \mathbb{F}_q$, with the condition that $h(r)=x$.  In our analysis of $f_{T_k, t_1, \ldots, t_{e(T_k)}}(x)$, our results will be independent of the $t_i$ (as long as they are nonzero), so we denote this function simply as $f_{T_k}(x)$. \bigskip

For a $\theta \in \oq{d}$, it is useful to define the function $\Gamma_{\theta, T_k}(w): \mathbb{F}_q^d \to \mathbb{R}$ as

\begin{equation}
    \Gamma_{\theta, T_k}(w)\ =\ \sum_{\substack{x, x' \in \mathbb{F}_q^d \\ x-\theta x' = w}}f_{T_k}(x)f_{T_k}(x').
\end{equation}

This function counts the number of pairs of congruent embeddings $h_1, h_2: V(T_k) \to E$ such that the transformation given by a rotation by $\theta$ and a translation by $w$ sends $h_1(r)$ to $h_2(r)$.  We prove some facts about $f_{T_k}$ and $\Gamma_{\theta, T_k}$ which are of use to us.

\begin{lemma}\label{funnerfacts}
Suppose $|E| \geqsim q^{k+\frac{d-1}{2}}$ and let $T_k$ be a $k$-weak simplex tree with $\ell$ simplices. Then
\begin{enumerate}
    \item $||f_{T_k}||_1\ \approx\ \eoq{\ell k + 1}{\ell \binom{k+1}{2}}$.
    \item $||f_{T_k}||_2\ \approx\ \eoq{\ell k + \frac{1}{2}}{\ell \binom{k+1}{2}}$.
    \item $||\Gamma_{\theta, T_k}||_1\ \approx\ \eoq{2\ell k+2}{2\ell\binom{k+1}{2}}$ and $\widehat{\Gamma}_{\theta, T_k}(0)\ \approx\ \eoq{2\ell k + 2}{d+2\ell  \binom{k+1}{2}}$.
\end{enumerate}
\end{lemma}
\begin{proof}
By definition, $f_{T_k}(x)$ counts the number of embeddings of $T_k$ rooted at $x$ in a fixed congruence class.  Hence, $||f_{T_k}||_1$ is the total number of trees in this congruence class class. From Theorem 1.12 of \cite{small2023}, this is on the order of $|E|^{\ell k + 1}q^{-\ell \binom{k+1}{2}}$ when $|E| \geqsim q^{k+\frac{d-1}{2}}$. \bigskip

To find $||f_{T_k}||_2$, note that $f^2_{T_k}(x)$ counts the number of pairs of embeddings of $T_k$ rooted at $x$ in a fixed congruence class.  Equivalently, $f^2_{T_k}(x)$ counts the number of embeddings of a larger $k$-weak simplex tree $T_k'$ rooted at $x$, where $T_k'$ is obtained by adjoining two copies of $T_k$ at $r$, and denoting $r$ as the root vertex of $T_k'$.  $T_k$ has $2\ell$ many $k$-simplices, so by Theorem 1.12 of \cite{small2023}, we have that for $|E| \geqsim q^{k+\frac{d-1}{2}}$, 
\begin{align}
    \sum_{x \in \mathbb{F}_q^d} f_{T_k}^2(x)\ \approx \ \eoq{2\ell k + 1}{2\ell \binom{k+1}{2}}.
\end{align}
Therefore, taking square roots, we get that
\begin{align}
    ||f_{T_k}||_2\ \approx \ \eoq{\ell k + \frac{1}{2}}{\ell  \binom{k+1}{2}}.
\end{align}

By the definition of the Fourier transform we have that
\begin{align}
    \widehat{\Gamma}_{\theta,T_k}(m)\ &=\ q^{-d}\sum_{w\in\mathbb{F}_q^d}\chi(-m\cdot w)\Gamma_{\theta,T_k}(w) \\
    &=\ q^{-d}\sum_{x,x'\in\mathbb{F}_q^d}\chi(-m\cdot (x-\theta x'))f_{T_k}(x)f_{T_k}(x')\ =\ q^d\widehat{f}_{T_k}(m)\overline{\widehat{f}_{T_k}(\theta^{-1}m)}.
\end{align}
We also have that for $|E|\ \geqsim\ q^{k+\frac{d-1}{2}}$,
\begin{align}\widehat{f}_{T_k}(0)\ =\ q^{-d}||f_{T_k}||_1\ \approx\ q^{-d}\left( \eoq{\ell k + 1}{\ell \binom{k+1}{2}} \right)\ =\ \eoq{\ell k + 1}{d+\ell  \binom{k+1}{2}}.
\end{align}
Therefore, we get the following results.
\begin{align}
    \widehat{\Gamma}_{\theta, T_k}(0)\ =\ q^d|\widehat{f}_{T_k}(0)|^2\ \approx\ \eoq{2\ell k+2}{d+2\ell\binom{k+1}{2}} \\
    ||\Gamma_{\theta, T_k}||_1\ =\ q^d\widehat{\Gamma}_{\theta, T_k}(0)\ \approx\ \eoq{2\ell k+2}{2\ell\binom{k+1}{2}}.
\end{align}

\end{proof}

In the coming sections, we need bounds for sums of the form 
\begin{equation}
    \sum_{\theta \in \oq{d}} \sum_{w \in \mathbb{F}_q^d} \lambda_{\theta}^{n}(w)\Gamma_{\theta, T_k}(w). 
\end{equation}
This bound looks slightly different in $\mathbb{F}_q^2$ than in higher dimensions, due to our ability to bound the finite field analogue of the Mattila integral better in two dimensions.

\begin{lemma}\label{hammerin2d}
    Let $T$ be a 1-weak simplex tree with $\ell$ simplices (i.e. a standard graph-theoretic tree with $\ell$ edges) and let $E \subset \mathbb{F}_q^2$ for $q \equiv 3 \pmod 4$.  If $|E| \geqsim q^{\frac{4n}{2n+1}}$ for some $n \geq 1$, then we have that

    \begin{equation}
        \sum_{\theta, w} \lambda_{\theta}^{n}(w)\Gamma_{\theta, T}(w)\ \lesssim \ \eoq{2\ell+2n+2}{2\ell +2n-1}.
    \end{equation}
\end{lemma}

\begin{lemma}\label{hammerindd}
    Let $T$ be a $k$-weak simplex tree with $\ell$ simplices, for $1 \leq k < \frac{d+1}{2}$, and let $E \subset \mathbb{F}_q^d$.  If $|E| \gtrsim q^{\max(\frac{dn+1}{n+1}, \, k + \frac{d-1}{2})}$ for some $n \geq 1$, we have that

    \begin{equation}
        \sum_{\theta, w} \lambda_{\theta}^{n}(w)\Gamma_{\theta, T_k}(w)\ \lesssim \ \eoq{2\ell k +2n +2}{2\ell\binom{k+1}{2} + d(n+1) - \binom{d+1}{2}}.
    \end{equation} 
\end{lemma}

The proofs of these two Lemmas are contained in the subsequent subsections.

\subsection{Proof of Lemma \ref{hammerin2d}}

We first prove a useful bound, which is a generalization of Lemma 3.10 of \cite{preprint}. Recall that in the setting of Lemma \ref{hammerin2d} we have $d = 2$, so our 1-weak simplex tree $T_k$ is just a standard graph-theoretic tree $T$.

\begin{lemma}\label{Gammaminuszerofourier2}
    For $q \equiv 3\pmod 4$, $E \subset \mathbb{F}_q^2$ and $|E| \gtrsim q^{3/2}$,
    \begin{equation}
        \sum_{\theta, w} \left( \Gamma_{\theta, T}(w) - \widehat{ \Gamma }_{\theta, T}(0)\right)^2\ \lesssim\ \eoq{4\ell+\frac{5}{2}}{4\ell-1}.
    \end{equation}
\end{lemma}

\begin{proof}
We first apply Parseval's theorem.

\begin{equation}
    \sum_{\theta, w} \left( \Gamma_{\theta, T}(w) - \widehat{\Gamma}_{\theta, T}(0) \right)^2\ =\ q^2 \sum_{\theta} \sum_{m \neq 0} \left| \widehat{\Gamma}_{\theta, T}(m) \right|^2.\label{firstlinehere}
\end{equation}
We note that

\begin{align}
    \widehat{\Gamma}_{\theta, T}(m)\ =\ q^{-2}\sum_{w}\chi(-m \cdot w) \Gamma_{\theta, T}(w)\ &=\ q^{-2}\sum_{u, u' \in \mathbb{F}_q^2} \chi(-m \cdot (u - \theta u')) f_T(u) f_T(u') \\ 
    &=\ q^{2}\widehat{f_T}(m)\overline{\widehat{f_T}(\theta^{-1}m)} \label{eqn:Gammaexp}.
\end{align}
Plugging the expression \eqref{eqn:Gammaexp} for $\widehat{\Gamma}_{\theta, T}(m)$ into \eqref{firstlinehere} gives us
\begin{align}
    \sum_{\theta, w} \left( \Gamma_{\theta, T}(w) - \widehat{\Gamma}_{\theta, T}(0) \right)^2\  =\ q^6 \sum_{\theta} \sum_{m \neq 0} |\widehat{f_T}(m)|^2 |\widehat{f_T}(\theta^{-1}m)|^2. \label{mattilad=2}
\end{align}

We sum first in $\theta$. Defining $||m|| = t$ and $h(m) = \overline{\widehat{f}_T(m)}$, we get that

\begin{align}
    \sum_{\theta} |\widehat{f_T}(\theta^{-1}m)|^2\ \approx\ \sum_{||\ell|| = t} | \widehat{f_T}(\ell)|^2\ &=\ \sum_{\ell} \widehat{f_T}(\ell) S_t(\ell) h(\ell) \\ &=\ q^{-2}\sum_{\ell}\sum_{w} \chi(-\ell \cdot w) f_T(w) S_{t}(\ell) h(\ell) \\
    &=\ \sum_{\ell}f_T(\ell) \widehat{hS_t}(\ell) 
   \ =\ ||f_T||_{{4/3}} ||\widehat{hS_t}||_{4},
\end{align}
where the last equality follows from an application of H{\"o}lder's inequality. We recall that by Lemma \ref{preprintstuff}, $||\widehat{hS_t}||_{4} \lesssim q^{-3/2}||hS_t||_{2}$ and we use the log-convexity of $L^p$ norms to bound $||f_T||_{{4/3}}$.

\begin{equation}
    ||f_T||_{{4/3}}\ \leq\ (||f_T||_{1})^{1/2} (||f_T||_{2})^{1/2}.
\end{equation}
These give

\begin{equation}2||hS_t||_{2}^2\ =\ \sum_{\theta} |\widehat{f_T}(\theta^{-1}m)|^2\ \lesssim\ q^{-3/2}||hS_t||_{2}(||f_T||_{1})^{1/2} (||f_T||_{2})^{1/2}, 
\end{equation}
from which we get that
\begin{align}
    \sum_{\theta} |\widehat{f_T}(\theta^{-1}m)|^2 \ \lesssim\ q^{-3}||f_T||_{1}||f_T||_{2}.
\end{align}
We calculate the remaining sum over $m$ in \eqref{mattilad=2}.
\begin{equation}
    \sum_{m \in \mathbb{F}_q^2} |\widehat{f_T}(m)|^2\ =\ q^{-2}  \sum_{x \in \mathbb{F}_q^2} f(x)^2\ =\ q^{-2} ||f_T||_{2}^2.
\end{equation}
Putting this all together, we conclude that
\begin{equation}
    \sum_{\theta, w} \left( \Gamma_{\theta, T}(w) - \widehat{\Gamma}_{\theta,T}(0) \right)^2 \ \lesssim\ q^6 \cdot q^{-2} ||f_T||_{2}^2 \cdot q^{-3}||f_T||_{1}||f_T||_{2}\ =\ q ||f_T||_{1} ||f_T||^3_{2}.
\end{equation}
For $|E| \gtrsim q^{3/2}$, we may plug in our bounds for $||f_T||_1$ and $||f_T||_2$ from Lemma \ref{funnerfacts}, completing the proof. \bigskip
\end{proof}

We are now ready to prove Lemma \ref{hammerin2d} by induction on $n$.  Let $A_n = \sum_{\theta, w} \lambda^n_\theta(w) \Gamma_{\theta, T}(w)$.  We first find a bound on $A_1$. By Cauchy-Schwarz, we have that 
\begin{equation}            \sum_{\theta, w}\lambda_\theta(w)\Gamma_{\theta, T}(w)\ \leq\ \left(\sum_{\theta, w}\lambda_{\theta}^2(w) \right)^{1/2} \left(\sum_{\theta, w} \Gamma_{\theta, T}^2(w) \right)^{1/2}.
\end{equation}

By Proposition \ref{lambdathetafacts2}, we have that $\sum_{\theta, w} \lambda_\theta^2(w) \lesssim |E|^4q^{-1}$ whenever $|E| \gtrsim q^{4/3}$. To bound $\sum_{\theta, w} \Gamma_{\theta, T}^2(w)$, we apply Lemma \ref{aimlemma39} and plug in our bound from Lemma \ref{Gammaminuszerofourier2}: for $|E| \gtrsim q^{3/2}$, we obtain

\begin{align}
    \sum_{\theta, w} \Gamma_{\theta, T}^2(w)\ &\lesssim\ \sum_{\theta}\eoq{4\ell+4}{4\ell+2} + \sum_{\theta, w}\left(\Gamma_{\theta, T}(w) - \widehat{\Gamma}_{\theta, T}(0)\right)^2\ \\  &\lesssim\ \eoq{4\ell+4}{4\ell+1} + \eoq{4\ell+\frac{5}{2}}{4\ell-1}. \label{basecase2dominate}
\end{align}
We see that the first term in \eqref{basecase2dominate} dominates the second term at $|E| \geqsim q^{4/3}$.  Putting everything together, we have for $|E| \gtrsim q^{4/3}$,
\begin{equation}
\sum_{\theta, w}\lambda_\theta(w)\Gamma_{\theta, T}(w)\ \lesssim\ \eoq{2\ell + 4}{2\ell+ 1}.
\end{equation}

This completes the base case.  Now assume that $n \geq 2$, and assume the inductive hypothesis that for $|E| \gtrsim q^{4(n-1)/(2(n-1)+1)}$,
\begin{equation}
    A_{n-1}\ \lesssim\ \eoq{2\ell + 2(n-1)+2}{2\ell + 2(n-1)-1}
\end{equation}
Since $\lambda_{\theta}^n(w)\Gamma_{\theta, T}(w)$ is positive by definition and $\widehat{\lambda}_\theta(0) = |E|^2q^{-2}$, we may write
\begin{align}
    \sum_{\theta, w} \lambda_{\theta}^n(w) \Gamma_{\theta, T}(w)\
    & \leq\  \left| \sum_{\theta, w} \lambda_{\theta}^{n-1}(w)\left( \lambda_{\theta}(w) - \frac{|E|^2}{q^2} \right)(\Gamma_{\theta, T}(w) - \widehat{\Gamma}_{\theta, T}(0)) \right|\nonumber \\
    &\hspace{0.5cm}+\ \sum_{\theta, w} \lambda_{\theta}^{n}\widehat{\Gamma}_{\theta, T}(0) + \frac{|E|^2}{q^2}A_{n-1} \\
    &=:\ I + II + III.
\end{align}
We bound $I$ by pulling absolute values inside the sum, taking the supremum bound on $\lambda_{\theta}^{n-1}(w)$, and applying Cauchy-Schwarz to the remaining terms, as follows:  
\begin{align}
    I\ &\leq\  |E|^{n-1}\left(\sum_{\theta, w}\left( \lambda_{\theta}(w) - \widehat{\lambda}_\theta(0)\right)^2\right)^{1/2} \left(\sum_{\theta, w}\left( \Gamma_{\theta, T}(w) - \widehat{\Gamma}_{\theta, T}(0)\right)^2\right)^{1/2} \\  &\lesssim\ \eoq{2\ell+\frac{3}{2}+n}{2\ell-1}.
\end{align}
The bound above follows from plugging in Proposition \ref{lambdathetafacts2} and Lemma \ref{Gammaminuszerofourier2}.  For $|E| \geqsim q^{4n/(2n+1)}$, we have that
\begin{equation}
  \eoq{2\ell+n + \frac{3}{2}}{2\ell-1} \ \lesssim \ \eoq{2\ell+2n+2}{2\ell+2n-1}.
\end{equation}

From Proposition \ref{lambdathetafacts2} and Lemma \ref{funnerfacts}, we get that for $|E| \gtrsim q^{4(n-1)/(2(n-1)+1)}$,
\begin{equation}
    II\ =\ \eoq{2\ell + 2}{2\ell + 2} \sum_{\theta, w}\lambda_\theta^n(w)\ \lesssim\ \eoq{2\ell + 2n+2}{2\ell + 2n-1}.
\end{equation}

Finally, we note that for $|E| \gtrsim q^{4(n-1)/(2(n-1)+1)}$,
\begin{equation}
    III \ = \ \frac{|E|^2}{q^2}A_{n-1} \ \lesssim \ \eoq{2\ell+2n+2}{2\ell+2n-1}.
\end{equation}
Combining these bounds completes the proof.

\subsection{Proof of Lemma \ref{hammerindd}}

Before proceeding to the proof of Lemma \ref{hammerindd}, we prove a $d$-dimensional analogue of the bound in Lemma \ref{Gammaminuszerofourier2}.  This is essentially a generalization of Lemma 4 in \cite{alexmcdonald}. 

\begin{lemma}\label{Gammaminuszerofourierd} For $E \subset \mathbb{F}_q^d$ and $|E| \gtrsim q^{k + (d-1)/2}$
    $$\sum_{\theta, w}(\Gamma_{\theta, T_k}(w) - \widehat{ \Gamma}_{\theta, T_k}(0) )^{2}\ \lesssim \ |E|^{4\ell k+2} q^{d + {\binom{d - 1} {2}}-4\ell\binom{k+1}{2}}. $$
\end{lemma}
\begin{proof} We see that 
    \begin{align}
    \sum_{\theta, w}(\Gamma_{\theta, T_k}(w) - \widehat{ \Gamma}_{\theta, {T_k}}(0) )^{2}\ &=\ q^d \sum_{\theta, m \neq 0} |\Gamma_{\theta, {T_k}}(m)|^2 \\
    &=\ q^{3d} \sum_{\theta, m \neq 0} | \widehat{ f }_{T_k}(m) |^{2} | \widehat{ f }_{T_k}(\theta^{-1}m) |^{2} \\
    &=\  q^{3d} \sum_{ m \neq 0} | \widehat{ f }_{T_k}(m)  |^{2} \sum_{\theta}|  \widehat{ f }_{T_k}(\theta^{-1}m) |^{2} \\
    &=\ q^{3d} \sum_{t} \sum_{\substack{m \neq 0 \\ ||m|| = t}} | \widehat{ f }_{T_k}(m) |^{2} \sum_{\substack{n \neq 0 \\ || n || = t}} | \widehat{ f }_{T_k}(n) |^{2} \sum_{\substack{\theta \\ \theta m = n}} 1. \label{beforemattilad}
\end{align}
Note that for $\theta, \phi \in \oq{d}$, we have that $\theta m = n$ and $\phi m = n $ if and only if $\phi^{-1} \theta $ fixes $m$.  Therefore, the sum over $\theta$ occurs the size of the stabilizer of $m$ many times, which is just $|O_{d-1}(\mathbb{F}_{q})|$ since $m, n \neq 0$. Therefore, we may write \eqref{beforemattilad} as
\begin{equation}
     q^{3d} |O_{d-1}(\mathbb{F}_{q})| \sum_{t} \left( \sum_{\substack{m \neq 0 \\ ||m|| = t}} | \widehat{ f }_{T_k} (m) |^{2}\right)^2. \label{mattilad}
\end{equation}
We write the square of sums as the product of two sums.  We dominate one of these sums by dropping all conditions on $m$ by positivity:

\begin{equation}
    \sum_{\substack{m \neq 0 \\ ||m|| = t}} | \widehat{ f }_{T_k} (m) |^{2}\ \leq\ \sum_m | \widehat{ f }_{T_k} (m) |^{2}\ =\ q^{-d} ||f_{T_k}||_{2}^2,
\end{equation}
where the last equality follows from Parseval.  We dominate the rest of \eqref{mattilad} by removing the non-zero condition:

\begin{equation}
    \sum_t \sum_{\substack{m \neq 0 \\ ||m|| = t}} | \widehat{ f }_{T_k} (m) |^{2}\ \leq\ \sum_m | \widehat{ f }_{T_k} (m) |^{2}\ =\ q^{-d} ||f_{T_k}||_{2}^2.
\end{equation}

All in all, we get that 
\begin{equation}
    \sum_{\theta, w}(\Gamma_{\theta, {T_k}}(w) - \widehat{ \Gamma}_{\theta, {T_k}}(0) )^{2}\ \leq\ q^{d}|O_{d-1}(\mathbb{F}_{q})| ||f_{T_k}||_2^{4}.
\end{equation}  

Plugging in $||f_{T_k}||_{2}$ from Lemma \ref{funnerfacts} and the size of $|O_{d-1}(\mathbb{F}_q)|$ from Proposition \ref{sizeoforthogonalgroup} gives the desired result.\bigskip
\end{proof}

\begin{proof}[Proof of Lemma \ref{hammerindd}]
    
    As in the proof of Lemma \ref{hammerin2d}, we proceed by induction on $n$. Denote by $A_n$ the sum $\sum_{\theta, w} \lambda^n_\theta(w) \Gamma_{\theta, {T_k}}(w)$.  We first obtain a bound for $A_1$. By Cauchy Schwarz, we have that 
    \begin{equation}
        \sum_{\theta, w}\lambda_\theta(w)\Gamma_{\theta, {T_k}}(w)\ \leq\ \left(\sum_{\theta, w}\lambda_{\theta}^2(w) \right)^{1/2} \left(\sum_{\theta, w} \Gamma_{\theta, {T_k}}^2(w) \right)^{1/2}.
    \end{equation}    
    Applying Lemma \ref{aimlemma39} to the $\Gamma$ sum and plugging in Lemma \ref{Gammaminuszerofourierd} for $|E| \gtrsim q^{k + (d-1)/2}$ gives us that
    \begin{align}
        \sum_{\theta, w} \Gamma_{\theta, {T_k}}^2(w) \ \lesssim\ q^{{\binom{d}{2}}} \eoq{4\ell k+4}{4\ell\binom{k+1}{2}+d} + \sum_{\theta, w}(\Gamma_{\theta, {T_k}} (w)-\widehat{\Gamma}_{\theta, {T_k}}(0))^2 \\ \ \lesssim\ q^{\binom{d} {2}}\eoq{4\ell k+4}{4\ell\binom{k+1}{2}+d} + |E|^{4\ell k+2} q^{d + {\binom{d - 1} {2}}-4\ell\binom{k+1}{2}}. \label{whosgonnadominatedbase}
    \end{align}
    By a straightforward computation, we see that the first term dominates in \eqref{whosgonnadominatedbase} for $|E| \gtrsim q^{(d+1)/2}$.  After plugging in the appropriate bounds from Proposition \ref{lambdathetafactsd}, we conclude that for $|E| \gtrsim q^{\max((d+1)/2, \, \,  k+(d-1)/2)}$,
    \begin{align}
        A_1\ \lesssim\ \left(|E|^{4} q^{{\binom{d+1}{2}}-2d}\right)^{1/2} \left( \frac{|E|^{4\ell k+4}}{q^{4\ell\binom{k+1}{2}+d}} q^{{\binom{d}{2}}} \right)^{1/2}\ &=\ |E|^{2\ell k+4}q^{\frac{d^2 - 3d - 4\ell\binom{k+1}{2}}{2}} \\
        &=\ \frac{|E|^{2\ell k+4}}{q^{2\ell\binom{k+1}{2} + 2d - \binom{d+1}{2}}}.
    \end{align}

This concludes the base case.  Now, for $n \geq 2$ and $|E| \gtrsim q^{\max(\frac{d(n-1)+1}{n}, \, k + \frac{d-1}{2})}$, assume that
\begin{equation}
    A_{n-1} \ \lesssim \ \eoq{2\ell k + 2(n-1) +2 }{2\ell\binom{k+1}{2} + d((n-1)+1) - \binom{d+1}{2}}
\end{equation}
We break up $A_n$ similarly as before. \begin{align}
    A_n\ =\ \sum_{\theta,w}\lambda_\theta^n(w)\Gamma_{\theta, {T_k}}(w)\ &\leq\ \left| \sum_{\theta,w}\lambda_\theta^{n-1}(w)\left(\lambda_\theta(w)-\frac{|E|^2}{q^d}\right)(\Gamma_{\theta, {T_k}}(w)-\widehat{\Gamma}_{\theta,  {T_k}}(0))\right|\nonumber\\
    &\hspace{0.5cm}+\  \sum_{w,\theta}\lambda_\theta^n(w)\widehat{\Gamma}_{\theta, {T_k}}(0)+\frac{|E|^2}{q^{d}}A_{n-1} \\
     &=:\ I + II+ III.
\end{align}
We bound each term individually.  To bound $I$, we pull absolute values inside the sum, take the supremum of $\lambda_{\theta}^{n-1}$, and apply Cauchy Schwarz to the remaining terms to obtain
    \begin{align}
    I\ \leq\ |E|^{n-1}\left(\sum_{\theta, w} \left( \lambda_{\theta}(w) - \frac{|E|^2}{q^d} \right)^2 \right)^{1/2} \left(\sum_{\theta, w} \left(\Gamma_{\theta, {T_k}}(w) - \widehat{\Gamma}_{\theta, {T_k}}(0)\right)^2 \right)^{1/2}.
\end{align}
By Proposition \ref{lambdathetafactsd}, we have that $\sum_{\theta, w}\left(\lambda_{\theta}(w) - \frac{|E|^2}{q^d} \right)^2 \lesssim |E|^2q^{{\binom{d}{2}} + 1}$.  Bounding the $\Gamma$ sum with Lemma \ref{Gammaminuszerofourierd} for $|E| \gtrsim q^{k + (d-1)/2}$, taking square roots, and simplifying gives us for $|E| \gtrsim q^{k + (d-1)/2}$,

\begin{equation}
    I\ \leq\ |E|^{2\ell k + n+1} q^{\frac{d^2-d+2}{2} - 2\ell\binom{k+1}{2}}.
\end{equation}
By a straightforward computation, we see that for $|E| \gtrsim q^{\frac{dn+1}{n+1}}$,
\begin{equation}
    I \ \lesssim \ \eoq{2\ell k + 2n +2}{2\ell\binom{k+1}{2}+ d(n+1) - \binom{d+1}{2}}.
\end{equation}

Next, we bound $II$. From Proposition \ref{lambdathetafactsd}, when $|E| \geqsim q^{\frac{dn - d +1}{n}}$, we have $\lambda^n_{\theta}(x) \leq |E|^{2n}q^{-dn + {\binom{d + 1}{2}}}$.  Plugging this in to our expression for $II$ and using our bounds from \ref{funnerfacts} gives us that

\begin{equation}
    II\ \lesssim\ \frac{|E|^{2\ell k+2n+2}}{q^{2\ell\binom{k+1}{2}+d(n+1) - \binom{d+1}{2}}}.
\end{equation}

Finally, the bound for $III$ quickly follows from the inductive hypothesis for \\ $|E| \gtrsim q^{\max(\frac{d(n-1)+1}{n}, \, k + \frac{d-1}{2})}:$

\begin{equation}
    III\ \lesssim\ \eoq{2}{d}\frac{|E|^{2\ell k+2(n-1)+2}}{q^{2\ell\binom{k+1}{2}+d(n) - \binom{d+1}{2}}}\ =\ \frac{|E|^{2\ell k+2n+2}}{q^{2\ell\binom{k+1}{2}+d(n+1) - \binom{d+1}{2}}}.
\end{equation}

\end{proof}

\section{Proof of main theorems}

In this section, we will prove Theorem \ref{maintheoremd}, \ref{extendsmall2023}, and \ref{maintheorem2}.  Recall from the discussion and remarks under \ref{congruenceclasstalk} and \ref{numembeddingsofsimplextrees} that for an embedding $h$ of a simplex tree $\TT$, we may assume that the embeddings of all simplices $S \in \TT$ (obtained by restricting the domain of $h$ to vertices of $S$) are nondegenerate.

\subsection{The setup}

For a given simplex tree $\TT$, we would like to find an $s$ such that $|E| \gtrsim q^s$ implies that $E$ contains a positive proportion of congruence classes of $\TT$ in $\mathbb{F}_q^d$.  Denote $\mathbb{D}$ as the set of congruence classes of $ \TT $. For $\delta \in \mathbb{D} $ ($\delta$ a congruence class of $\TT$), we define $\nu_{\TT}(\delta)$ as the number of embeddings of $\TT$ in $E$ that are in the congruence class $\delta$. We denote by $ \Delta(E) $ the size of the support of $\nu_{\TT}$, or in other words, the number of congruence classes of $\TT$ in $E$.  \bigskip

By an application of Cauchy-Schwarz, we have that 

\begin{equation}
    \Delta(E)\ \geq \ \frac{(\sum_{\delta \in \mathbb{D}}\nu_{\TT}(\delta))^{2}}{\sum_{\delta \in \mathbb{D}} \nu_{\TT}^{2}(\delta)}.
\end{equation}

The sum in the numerator runs over all possible $|V(\TT)|$-tuples in $E$ (where $V(\TT)$ denotes the vertex set of $\TT$). Hence we have that

\begin{equation}
    \Delta(E)\ \geqsim \ \frac{|E|^{2|V(\TT)|}}{\sum_{\delta \in \mathbb{D}} \nu_{\TT}^{2}(\delta)}.
\end{equation}

To obtain a lower bound on $\Delta(E)$, it suffices to get an upper bound on $\stsum$.  More precisely, to prove Theorems \ref{maintheoremd} and \ref{maintheorem2}, we need to show that for the $s$ in the theorem statements, $\Delta(E)$ is on the order of $ q^{c(\TT)}$ whenever $|E| \gtrsim q^s$.  By the setup above, we see that
\begin{align}
\Delta(E) \geqsim q^{c(\TT)} \quad \Longleftrightarrow \quad
\stsum \lesssim {\eoq{2|V(\TT)|}{c(\TT)}}.    
\end{align}

Our first goal is to find an expression of a more workable form for $\stsum$, the number of pairs of congruent embeddings of $\TT$. We build this expression recursively, by first counting the number of pairs of congruent embeddings of simplices, then sub-simplex trees, and finally the full simplex tree $\TT$. \bigskip

For $u, u' \in E$ and an $n$-simplex $S$, we begin by finding an expression for the number of pairs of congruent embeddings of $S$ in $E$ based at $u$ and $u'$, i.e. congruent embeddings $h_1, h_2: V(S) \to E$ of $n$-simplices into $E$ such that $h_1(v_0) = u$ and $h_2(v_0) = u'$.  The following exposition follows \cite{groupactions} closely, but our notation differs and these ideas will be developed further later, so we present it here. \bigskip

We define $\lambda_{\theta}(w)$ as in Definition \ref{lambdathetadef}.  This is the number of pairs $x, y \in E$ such that the `rigid~motion' $\rho(\theta, w)$, consisting of a rotation by $\theta \in \oq{d}$ followed by a translation by $w$, sends $y$ to $x$.  By the discussion under \ref{congruenceclasstalk}, two nondegenerate $n$-simplices $S_x = (x_0, \ldots, x_n)$ and $S_y = (y_0, \ldots, y_n)$ with vertices in $\mathbb{F}_q^d$ are congruent if and only if there exists a rigid motion $\rho(\theta, w)$ such that $x_i = \rho(\theta, w)y_i$ for all $i$. Therefore, we see that $\lambda^{n+1}_{\theta}(w)$ counts the number of pairs of congruent $n$-simplices sent to each other by the transformation $\rho(\theta, w)$ (as we need to find $n+1$ pairs of points such that $\rho(\theta, w)$ sends the first element of each pair to the second).  For the rest of this section, when the context is clear, a simplex $S$ living in $\mathbb{F}_q^d$ is actually an embedding of a simplex into $\mathbb{F}_q^d$.  \bigskip

For $u, u' \in \mathbb{F}_q^d$, we also frequently make use of the term $\lambda_{\theta}(u - \theta u')$, which counts pairs $x, x'$ such that the transformation given by a rotation by $\theta$ and translating the image of $u'$ to $u$, also transforms $x'$ to $x$.  Therefore, $\lambda_{\theta}^n(u - \theta u')$ counts the number of pairs of embeddings of $n$-simplices, one of which contains $u$ as a vertex and the other of which contains $u'$, that are sent to each other by first rotating by $\theta$ and then translating the image of $u'$ to $u$ (the exponent is $n$ and not $n+1$ as $u, u'$ are already fixed). \bigskip

For $\theta, \phi \in \oq{d}$ and $n$-simplices $S_x = (x_0, \ldots, x_n)$ and $S_y = (y_0, \ldots, y_n)$ the transformations $\rho(\theta, x_0 - \theta y_0)$ and $\rho(\phi, x_0 - \phi y_0)$ both send $y_i$ to $x_i$ for all $i$ if and only if $\theta$ and $\phi$ differ by an element in the stabilizer of the $n$-simplex $S_{y-y_0} = (0, y_1 - y_0, \ldots, y_n - y_0)$. Since for congruent $n$-simplices $S_x$ and $S_y$, the corresponding pinned simplices $S_{x-x_0}$ and $S_{y-y_0}$ have conjugate stabilizer groups, we may define $\Stab(S_x)$ to be the common size of the stabilizer $S_{y-y_0}$ of simplices $S_y$ congruent to $S_x$.  We see that for $u, u' \in E \subset \mathbb{F}_q^d$, the sum 
\begin{equation}
    \sum_{\theta \in \oq{d}} \lambda_{\theta}^{n}(u - \theta u')
\end{equation}
counts each pair of congruent embeddings of $n$-simplices $(S_x, S_y)$ in $E$ exactly $\Stab(S_x)$ many times.  Let $\Stab(n)$ denote the minimum size of the stabilizer of a $n$-simplex in $\mathbb{F}_q^d$.  We have the following fact from \cite{groupactions} for $n \leq d$, which we easily extend to $n \geq d$:


\begin{prop} \label{stabsizelemma}
    The minimum of the stabilizer of a dimension $k$ simplex is
    \begin{equation*}
       \mathrm{Stab}(n)\ \approx \ \begin{cases}
         \oqsize{n} \ \approx \ q^{\binom{d-n}{2}} & n < d-1 \\
         1 & n \geq d-1.
    \end{cases} 
    \end{equation*}
\end{prop}
The case for $n \geq d$ follows trivially from the fact that a simplex $S_x = (x_0, \ldots, x_n)$ of dimension $d$ or larger, the vectors $x_1 - x_0, \ldots, x_n - x_0$ may span all of $\mathbb{F}_q^d$ hence the stabilizer is minimally 1. \bigskip

We note that since we can assume that all simplices are nondegenerate, the stabilizers of our simplices are always of minimal size.  Therefore, the expression

\begin{equation}\label{eqn:simplexDdef}
    D_S(u, u')\ =\ \frac{1}{\Stab(\dim(S))}\sum_{\theta \in \oq{d}} \lambda_{\theta}^{\dim(S)}(u - \theta u')
\end{equation}

counts each pair of congruent embeddings of $S$ in $E$ based at $u$ and $u'$ at on the order of one time, since without the stabilizer term out front we count each pair of congruent embeddings on the order of $\Stab(\dim(S))$ times. \bigskip

We now use the above methods to obtain an expression for the number of pairs of congruent embeddings $h_1, h_2$ of a free rooted simplex tree $(\TT, S, v_0) $ based at $u$ and $u'$ (i.e.  $h_1(v_0) = u$ and $h_2(v_0) = u'$). 

\begin{prop}\label{bigsumforFRST}
    Suppose we have a free rooted simplex tree $(\TT, S, v_0)$, for $S \in \TT$ and $v_0 \in S$ a free vertex.  Define $D_{(\TT, S, v_0)}: E^2 \to \mathbb{R}$ recursively as

\begin{equation*}
    D_{(\TT, S, v_0)}(u, u')\ =\ \frac{1}{\Stab(\dim(S))} \sum_{\theta \in \oq{d}} \, \prod_{\substack{ v \in V(S)  \\ v \neq v_0}} \, \sum_{\substack{x, x' \in E \\ x - \theta x' \\ = u - \theta u'}} \, \, \prod_{B \in \mathcal{B}(S, v)}D_{B}(x, x'). \label{freesimplextreesum} 
\end{equation*}
    Then for each $u, u' \in E$, $D_{(\TT, S, v_0)}(u, u')$ is on the order of the number of pairs of congruent embeddings of $(\TT, S, v_0)$ in $E$ based at $u$ and $u'$. 
\end{prop}

\begin{proof}
     If $(\TT, S, v_0)$ is a simplex, then $|\mathcal{B}(S, v)| = 0$.  Defining the empty product over $\mathcal{B}(S, v) = \emptyset$ to be 1, we see that $D_{(\TT, S, v_0)}(u, u')$ reduces to $D_S(u, u')$, and by the discussion above, $D_S(u, u')$ is on the order of the number of pairs of congruent embeddings of $(\TT, S, v_0)$ based at $u, u'$. \bigskip
     
     We now proceed by induction on the number of simplices in $(\TT, S, v_0)$. We can obtain an expression on the order of the number of pairs of congruent embeddings of $(\TT, S, v_0)$ in $E$ based at $u$ and $u'$ as follows: \begin{enumerate}
        \item Fix a $\theta$, and consider all embeddings of simplices $h_1, h_2: V(S) \to E$ based at $u, u'$ such that the rotation $\theta$ followed by the translation sending $\theta u'$ to $u$ also sends $h_2(v)$ to $h_1(v)$ for all $v \in V(S) \setminus v_0$.
        
        \item For each such fixed pair of embeddings of $S$, count the number of ways this pair $h_1, h_2$ can be extended to a pair of congruent embeddings of $(\TT, S, v_0)$ in $E$.

        \item Sum over all $\theta$ and divide the final sum by $\Stab(\dim(S))$.
    \end{enumerate} 

    First, we explain why
    \begin{equation}
        \prod_{\substack{v \in V(S) \\ v \neq v_0}}  \, \, \sum_{\substack{x, x' \in E \\ x - \theta x' = u - \theta u'}} \, \, \prod_{B \in \mathcal{B}(S, v)}D_{B}(x, x')
    \end{equation}
    corresponds to steps 1 and 2.  For notational convenience, let $\dim(S) = n$, so that $S$ has $n+1$ vertices. For a fixed $\theta$, we consider all congruent embeddings $h_1, h_2: V(S) \to E$ such that the transformation $\rho(\theta, u - \theta u')$ sends $h_2(v)$ to $h_1(v)$ for all $v \in V(S)$.  Each pair of these congruent embeddings $h_1, h_2$ corresponds to $n$ pairs of points $(a_i, b_i) \in \mathbb{F}_q^2$ (one pair for each $v \neq v_0$, as $u$ and $u'$ are already fixed) such that $\rho(\theta, u - \theta u')b_i = a_i$. For a fixed vertex $v \in V(S) \setminus v_0$, the sum over $x,x' \in E$ ranges over all possible candidates for $h_1(v), h_2(v)$ with respect to a fixed transformation. For each such $x, x'$ and a fixed branch $B$ of $(\TT, S, v_0)$ at the vertex $v$, by the inductive hypothesis, $D_B(x,x')$ is on the order of the number of pairs of congruent embeddings of $B$ in $E$ based at $x$ and $x'$.  The product over all branches $B$ at the vertex $v$ is on the order of the number of pairs of congruent embeddings based at $x$ and $x'$ of the structure obtained by joining all branches in $\mathcal{B}(S, v)$ at $v$. Taking the product over all vertices $v \in V(S)$ distinct from $v_0$ yields an expression on the order of the number of pairs of congruent embeddings of $(\TT, S, v_0)$ extending $h_1,h_2$. \bigskip

    We now argue step 3.  Fixing a pair of based congruent embeddings $h_1, h_2: V(S) \to E$, of $S$, for each $\phi~\in~\mathrm{Stab}(h_2(S))$ and each $\theta$ such that $\rho(\theta, u - \theta u')$ sends $h_2(S)$ to $h_1(S)$, we also have that $\rho(\theta \phi, u - \theta \phi u')$ sends $h_2(S)$ to $h_1(S)$.  Therefore, once we sum over all $\theta$, we overcount each pair of embeddings of $S$ on the order of $\mathrm{Stab}(h_2(S)) = \mathrm{Stab}(\dim(S))$ many times (as $h_2(S)$ is nondegenerate).  Therefore, Steps 1 and 2 overcount each pair of congruent embeddings of $\mathcal{T}$ by the same factor, and we divide by $\mathrm{Stab}(\dim(S))$ to complete the proof of the Proposition.
    \bigskip
\end{proof}


Finally, we construct an expression for the number of pairs of congruent embeddings of a general rooted simplex tree $(\TT, S_0)$.

\begin{prop}
    Consider a rooted simplex tree $(\TT, S_0)$, and define

\begin{equation*}
    R_{(\TT, S_0)}\ =\  \frac{1}{\Stab(\dim(S_0))} \sum_{w \in \mathbb{F}_q^d} \, \, \sum_{\theta \in \oq{d}} \, \,  \prod_{v \in V(S_0)}  \, \, \sum_{\substack{x, x' \in E \\ x - \theta x' = w}} \, \, \prod_{B \in \mathcal{B}(S_0, v)}D_{B}(x, x'). \label{bigsumddim}
\end{equation*}
    Then we have that 
    \begin{equation*}
        \stsum\ \approx \ R_{(\TT, S_0)}.
    \end{equation*}
\end{prop}

\begin{proof}
    Note that $\stsum$ counts the number of pairs of congruent embeddings of $(\TT, S_0)$ in $E$.  We can obtain an expression on the order of this quantity as follows: \begin{enumerate}
        \item Fix a $\theta, w$, and consider all embeddings of simplices $h_1, h_2: V(S_0) \to E$ such that $h_1(S_0) = \rho(\theta, w)h_2(S_0)$.

        \item For each fixed pair of congruent embeddings of the root $S_0$ corresponding to $\rho(\theta, w)$, count the number of ways this pair of embeddings $h_1, h_2$ can be extended to a pair of congruent embeddings of $(\TT, S_0)$ in $E$.

        \item Sum over all transformations $\rho(\theta, w)$, and divide our final sum by $\Stab(\dim(S_0))$.
    \end{enumerate} 

    The quantity
    \begin{equation}
        \prod_{v \in V(S_0)}  \, \, \sum_{\substack{x, x' \in E \\ x - \theta x' = w}} \, \, \prod_{B \in \mathcal{B}(S_0, v)}D_{B}(x, x')
    \end{equation}
    corresponds to steps 1 and 2.  
    For notational convenience, let $\dim(S_0) = n$.  For a fixed $\theta$, we consider all congruent embeddings $h_1, h_2: V(S) \to E$ such that the transformation $\rho(\theta, w)$ sends $h_2(v)$ to $h_1(v)$ for all $v \in V(S)$.  Each pair of these embeddings $h_1, h_2: V(S) \to E$ corresponds to $n+1$ pairs of points $(a_i, b_i)$ such that $\rho(\theta, w)b_i = a_i$ for all $i$.  For a $v \in V(S_0)$, the sum over $x,x' \in E$ ranges over all possible candidates for $h_1(v), h_2(v)$ with respect to a fixed $\rho(\theta, w)$.  We weight each candidate as in Proposition \ref{bigsumforFRST}, and taking the product over all $v \in V(S_0)$ yields a quantity on the order of the number of pairs of congruent embeddings of $(\TT, S_0)$ extending $h_1,h_2$. \bigskip

    Now we argue Step 3 by a similar argument as in Proposition \ref{bigsumforFRST}.  Fixing a pair of congruent embeddings $h_1, h_2: V(S_0) \to E$, for each $\phi \in \mathrm{Stab}(h_2(S_0))$ and each pair $\theta, w$ such that $\rho(\theta, w)$ sends $h_2(S_0)$ to $h_1(S_0)$, we also have that $\rho(\theta \phi, w)$ sends $h_2(S_0)$ to $h_1(S_0)$.  Therefore, the sum over all $\theta$ and $w$ counts $h_1$, $h_2$ exactly $\mathrm{Stab}(h_1(S_0)) = \Stab(\dim(S_0))$ many times.  Therefore, after summing over all $\theta$ and dividing by $ \Stab(\dim(S_0))$, we see that each pair of congruent embeddings of $\TT$ is counted on the order of one time. \bigskip
\end{proof}

Therefore, in order to prove both Theorem \ref{maintheoremd} and \ref{maintheorem2}, it is equivalent to show the following:

\begin{prop} \label{thebigguy}
    For a rooted simplex tree $(\TT, S_0)$, and for any $1 \leq k < \frac{d+1}{2}$, define
    \begin{equation}
        N_k \ = \ k+\sum_{\substack{S \in \TT \\ \dim(S) > k}} (\dim(S) -k).
    \end{equation}    
    Then for $s = \max \left(\frac{dN_k + 1}{N_k+1}, \, \, k + \frac{d-1}{2} \right) $, we have that whenever $|E| \gtrsim q^s$,
    \begin{equation}
        R_{(\TT, S_0)}\ \lesssim \ \eoq{2|V(\TT)|}{c(\TT)}.
    \end{equation}
    If we additionally impose the condition that $d=2$ and that $q \equiv 3 \pmod 4$, for $s = \frac{4N_1}{2N_1 + 1}$ we have that whenever $|E| \gtrsim q^s$ the above holds as well.
    
\end{prop}
    
From here, our goal will be to apply the Hadamard three-lines Theorem to bound $R_{(\TT, S_0)}$ effectively. 
 This geometrically corresponds to reducing the proof of Proposition \ref{thebigguy} for $\TT$ to proving Proposition \ref{thebigguy} for some collection of nicer simplex trees. 

\subsection{Two geometric operations}
In this subsection, we describe and prove the two operations we use to simplify our sum: branch shifting and simplex unbalancing.  Before we do so, we introduce some notation to make the statement and proof of these operations more clear.

\begin{definition}
    For a rooted simplex tree $(\TT, S_0)$, denote $\ppc(\TT, S_0)$ as the infimum over $s \in \mathbb{R}$ such that $R_{(\TT, S_0)} \lesssim E^{2|V(\TT)|}q^{-c(\TT)}$ (equivalently, $\stsum \lesssim E^{2|V(\TT)|}q^{-c(\TT)}$) whenever $|E| \gtrsim q^s$.
\end{definition}

\begin{lemma}[Branch Shifting]\label{branchshifting} Given a rooted simplex tree $(\TT, S_0)$, consider two simplices $S_1$ and $S_2$, that are not necessarily distinct from each other or the root vertex.  Further, suppose each of $S_1$, $S_2$ contain respective, necessarily distinct, child vertices $v_1$ and $v_2$.  Consider the rooted tree $(\TT_1, S_0)$ formed from $(\TT, S_0)$ by deleting all branches of $S_2$ at $v_2$ and duplicating all branches of $S_1$ at $v_1$. Similarly, consider the rooted tree $(\TT_2, S_0)$ formed from $(\TT, S_0)$ by deleting all branches of $S_1$ at $v_1$ and duplicating all branches of $S_2$ at $v_2$.  Then $\ppc(\TT, S_0) \leq \max(\ppc(\TT_1, S_0), \ppc(\TT_2, S_0))$ \end{lemma}



Note that the quantity $N_k$ in Proposition \ref{thebigguy} is the same for $\TT, \TT_1,$ and $\TT_2$.  Therefore, the above Lemma states that in order to prove Proposition \ref{thebigguy} for $\TT$, it suffices to prove Proposition \ref{thebigguy} for $\TT_1$ and $\TT_2$. An example of branch shifting shown in Figure \ref{branchshiftingfig}, with $S_0$ being the black 3-simplex, $S_1, S_2$ both the red triangle, $v_1$ being the parent node of the green branches, and $v_2$ being the parent node of the blue branch. \bigskip

\begin{figure}[hbt!] 
\centering
\tikzset{every picture/.style={line width=1pt}} 

\begin{tikzpicture}[x=0.75pt,y=0.75pt,yscale=-1,xscale=1]

\draw  [color={rgb, 255:red, 255; green, 0; blue, 0 }  ,draw opacity=1 ] (114.37,107) -- (139.74,152.51) -- (89,152.51) -- cycle ;
\draw  [color={rgb, 255:red, 74; green, 144; blue, 226 }  ,draw opacity=1 ] (185.57,177.31) -- (140.38,203.25) -- (139.74,152.51) -- cycle ;
\draw [color={rgb, 255:red, 74; green, 144; blue, 226 }  ,draw opacity=1 ]   (185.57,177.31) -- (192.67,221) ;
\draw [color={rgb, 255:red, 126; green, 211; blue, 33 }  ,draw opacity=1 ]   (89,152.51) -- (94.67,197) ;
\draw [color={rgb, 255:red, 126; green, 211; blue, 33 }  ,draw opacity=1 ]   (89,152.51) -- (58.67,185) ;
\draw  [color={rgb, 255:red, 255; green, 0; blue, 0 }  ,draw opacity=1 ] (481.37,106) -- (506.74,151.51) -- (456,151.51) -- cycle ;
\draw  [color={rgb, 255:red, 255; green, 0; blue, 0 }  ,draw opacity=1 ] (320.37,108) -- (345.74,153.51) -- (295,153.51) -- cycle ;
\draw [color={rgb, 255:red, 126; green, 211; blue, 33 }  ,draw opacity=1 ]   (295,153.51) -- (329.16,182.57) ;
\draw [color={rgb, 255:red, 126; green, 211; blue, 33 }  ,draw opacity=1 ]   (295,153.51) -- (294.48,197.96) ;
\draw [color={rgb, 255:red, 126; green, 211; blue, 33 }  ,draw opacity=1 ]   (295,153.51) -- (266.55,188.18) ;
\draw [color={rgb, 255:red, 126; green, 211; blue, 33 }  ,draw opacity=1 ]   (295,153.51) -- (250.55,153.77) ;
\draw   (121.06,47.95) -- (147.24,80.82) -- (114.37,107) -- (88.19,74.13) -- cycle ;
\draw [color={rgb, 255:red, 155; green, 155; blue, 155 }  ,draw opacity=1 ]   (122.06,48.95) -- (115.37,108) ;
\draw [color={rgb, 255:red, 155; green, 155; blue, 155 }  ,draw opacity=1 ]   (89.19,75.13) -- (148.24,81.82) ;
\draw   (488.06,46.95) -- (514.24,79.82) -- (481.37,106) -- (455.19,73.13) -- cycle ;
\draw [color={rgb, 255:red, 155; green, 155; blue, 155 }  ,draw opacity=1 ]   (455.19,73.13) -- (514.24,79.82) ;
\draw [color={rgb, 255:red, 155; green, 155; blue, 155 }  ,draw opacity=1 ]   (488.06,46.95) -- (481.37,106) ;
\draw   (327.06,48.95) -- (353.24,81.82) -- (320.37,108) -- (294.19,75.13) -- cycle ;
\draw [color={rgb, 255:red, 155; green, 155; blue, 155 }  ,draw opacity=1 ]   (294.19,75.13) -- (353.24,81.82) ;
\draw [color={rgb, 255:red, 155; green, 155; blue, 155 }  ,draw opacity=1 ]   (328.06,49.95) -- (321.37,109) ;
\draw  [color={rgb, 255:red, 74; green, 144; blue, 226 }  ,draw opacity=1 ] (527.37,199.36) -- (475.82,191.75) -- (506.74,151.51) -- cycle ;
\draw [color={rgb, 255:red, 74; green, 144; blue, 226 }  ,draw opacity=1 ]   (527.37,199.36) -- (494.67,226.67) ;
\draw  [color={rgb, 255:red, 74; green, 144; blue, 226 }  ,draw opacity=1 ] (558.29,159.13) -- (506.74,151.51) -- (537.66,111.28) -- cycle ;
\draw [color={rgb, 255:red, 74; green, 144; blue, 226 }  ,draw opacity=1 ]   (558.29,159.13) -- (553.67,194.67) ;

\draw (281,20) node [anchor=north west][inner sep=0.75pt]   [align=left] {$\displaystyle \mathcal{T}_{1}$ (modified)};
\draw (441,18) node [anchor=north west][inner sep=0.75pt]   [align=left] {$\displaystyle \mathcal{T}_{2}$ (modified)};
\draw (84,18) node [anchor=north west][inner sep=0.75pt]   [align=left] {$\displaystyle \mathcal{T} \ $(original)};

\end{tikzpicture}
\caption{An example of branch shifting.}
\label{branchshiftingfig}
\end{figure}
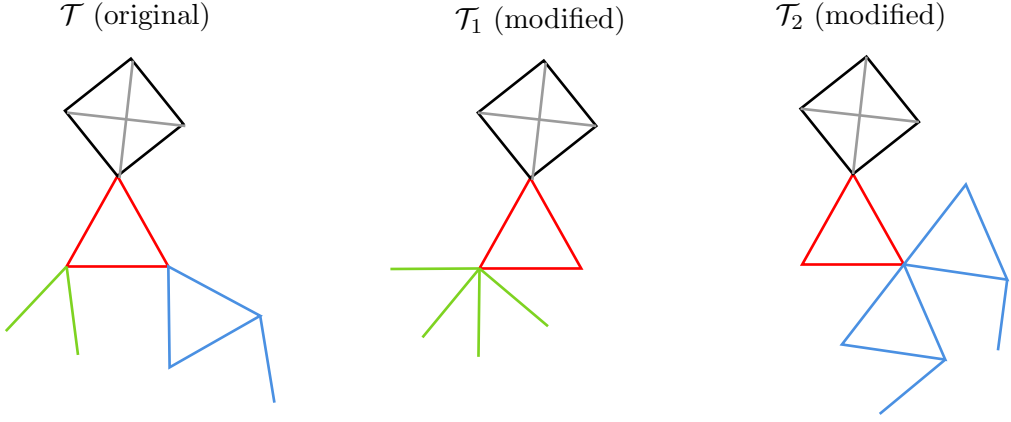

\begin{proof}
First, suppose that $S_1$ and $S_2$ are distinct from each other and from the root simplex, with parent vertices $v_{0, 1}$ and $v_{0, 2}$ respectively.  Then for $\ell \in \{ 1, 2 \}$, the expression $R_{(\TT, S_0)}$ contains the term

\begin{align}
    D_{\mathcal{P}(S_\ell)}(u, u')\ =\ \frac{1}{\Stab(\dim(S_\ell))} \sum_{\theta \in \oq{d}} \, \prod_{\substack{ v \in V(S_\ell)  \\ v \neq v_{0, \ell}}} \, \, \sum_{\substack{x, x' \in E \\ x - \theta x' \\ = u - \theta u'}} \, \, \prod_{B \in \mathcal{B}(S_\ell, v)}D_{B}(x, x').
\end{align}
Now consider the modified expression 

\begin{equation}
    D_{\mathcal{P}(S_\ell)}(u, u', a_{\ell})\ =\ \frac{1}{\Stab(\dim(S_\ell))} \sum_{\theta \in \oq{d}} \, \prod_{\substack{ v \in V(S_\ell)  \\ v \neq v_{0, \ell}}} \, \, \sum_{\substack{x, x' \in E \\ x - \theta x' \\ = u - \theta u'}} \, \, \left( \prod_{B \in \mathcal{B}(S_\ell, v)}D_{B}(x, x') 
    \right)^{G_\ell(v)}
\end{equation}

where $G_\ell: V(S_\ell) \to \mathbb{R}$ is defined as
\begin{equation} 
    G_\ell(v)\ =\ \begin{cases}
        1 & v \neq v_\ell \\
        a_\ell & v = v_\ell. 
    \end{cases}
    \label{Gdef}
\end{equation}

We now consider the function $\psi(a_1, a_2)$ defined as the expression for $R_{(\TT, S_0)}$ where we substitute all instances of $ D_{\mathcal{P}(S_\ell)}(u, u')$ with $D_{\mathcal{P}(S_\ell)}(u, u', a_{\ell})$.  Therefore, we see that \begin{equation}
    R_{(\TT, S_0)}\ =\ \psi(1, 1).
\end{equation}
It is easy to check that this $\psi$ satisfies the requirements for the special case of Hadamard three-lines in Corollary \ref{repeatedhadamard} (by an iterated use of the triangle inequality), so applying this gives us that 

\begin{equation}
    \psi(1, 1)\ \leq\ \psi(2, 0)^{\frac{1}{2}} \psi(0, 2)^{\frac{1}{2}}
\end{equation}
Note that the terms $\psi(2, 0)$ and $\psi(0, 2)$ are equal to $R_{(\TT_1, S_0)}$ and $R_{(\TT_2, S_0)}$ respectively, where $\TT_1$ and $\TT_2$ are as defined in the Lemma statement.  We also easily see that $2 c(\TT) = c(\TT_1) + c(\TT_2)$ and that $2 |V(\TT)| = |V(\TT_1)|+|V(\TT_2)|$.  Since for any $s > \max(\ppc(\TT_1),\ppc(\TT_2)) $, we have that $|E| \gtrsim q^s$ implies \begin{equation}
    \psi(2, 0) \psi(0, 2) \ \lesssim \ \eoq{2|V(\TT_1)| + 2|V(\TT_2)|}{c(\TT_1) + c(\TT_2)},
\end{equation}
we conclude that
\begin{equation}
    \psi(1, 1)\ \lesssim\ \eoq{2|V(\TT)|}{c(\TT)}.
\end{equation}

In the case where $S_1$ and $S_2$ are the same simplex and are still distinct from the root node, we denote both of them as $S$ for simplicity.  We denote the parent vertex of $S$ as $v_0$, and the two distinguished child vertices as $v_1$ and $v_2$.  Then the expression $R_{(\TT, S_0)}$ contains the term

\begin{align}
    D_{\mathcal{P}(S)}(u, u')\ =\ \frac{1}{\Stab(\dim(S))} \sum_{\theta \in \oq{d}} \, \prod_{\substack{ v \in V(S)  \\ v \neq v_{0}}} \, \, \sum_{\substack{x, x' \in E \\ x - \theta x' \\ = u - \theta u'}} \, \, \prod_{B \in \mathcal{B}(S, v)}D_{B}(x, x').
\end{align}
Now consider the modified expression 

\begin{align}
    &D_{\mathcal{P}(S)}(u, u', a_1, a_2) \nonumber \\   =\ &\frac{1}{\Stab(\dim(S))} \sum_{\theta \in \oq{d}} \, \prod_{\substack{ v \in V(S)  \\ v \neq v_{0}}} \, \, \sum_{\substack{x, x' \in E \\ x - \theta x' \\ = u - \theta u'}} \, \, \left( \prod_{B \in \mathcal{B}(S, v)}D_{B}(x, x') 
    \right)^{G_1(v) + G_2(v) - 1}
\end{align}
where $G_1$ and $G_2$ are defined as in \eqref{Gdef} above.  We define $\psi(a_1, a_2)$ as the expression for $R_{(\TT, S_0)}$ where we substitute each instance of $D_{\mathcal{P}(S)}(u, u')$ with $D_{\mathcal{P}(S)}(u, u', a_1, a_2)$.  We note that $R_{(\TT, S_0)} = \psi(1, 1)$.  Applying Hadamard three-lines and arguing as in the previous case completes the proof of the Lemma in this case. \bigskip

The cases where one or more of $S_1, S_2$ is the root of $\TT$ follow in an extremely similar fashion.  Suppose that $S_1$ is the root of $\TT$, then $R_{(\TT, S_0)}$ is of the form

\begin{equation}
     R_{(\TT, S_0)}\ =\ \sum_{w \in \mathbb{F}_q^d} \, \sum_{\theta \in \oq{d}} \prod_{\substack{v \in V(S_1)}} \sum_{\substack{x, x' \in E \\ x - \theta x' = w}} \prod_{B \in \mathcal{B}(S_1, v)}D_{B}(x, x'). 
\end{equation}
Define $\psi(a_1, a_2)$ as the modified expression
\begin{equation}
    \sum_{w \in \mathbb{F}_q^d}  \,\sum_{\theta \in \oq{d}} \,\prod_{\substack{v \in V(S_1)}} \, \sum_{\substack{x, x' \in E \\ x - \theta x' = w}} \left( \prod_{B \in \mathcal{B}(S_1, v)}D_{B}(x, x') 
    \right)^{G_1(v)}
\end{equation}
where we substitute each instance of $D_{\mathcal{P}(S_2)}(u, u')$ with $D_{\mathcal{P}(S_2)}(u, u', a_\ell)$.  From here, we apply Hadamard three-lines to $\psi(1, 1)$ and argue as in the previous cases.  Finally, in the case where $S_1 = S_2 = S_0$, we define $\psi(a_1, a_2)$ as the modified expression for $R_{(\TT, S_0)}$ below.
\begin{equation}
    \sum_{w \in \mathbb{F}_q^d} \, \sum_{\theta \in \oq{d}} \, \prod_{\substack{v \in V(S_1)}} \, \sum_{\substack{x, x' \in E \\ x - \theta x' = w}}  \left( \prod_{B \in \mathcal{B}(S_1, v)}D_{B}(x, x') 
    \right)^{G_1(v) + G_2(v) - 1}
\end{equation}
Then applying Hadamard three-lines to $\psi(1, 1)$ finishes this case.
\end{proof}

We now proceed to our second geometric operation:

\begin{lemma}[Simplex Unbalancing] \label{simplexunbalancing} Given a rooted simplex tree $(\TT, S_0)$, consider two distinct simplices $S_1$ and $S_2$ (one of which may be the root), each of dimension at least 2 and each with at least one free vertex.  Choose positive integers $k_1, k_2$ such that $k_i$ does not surpass the number of free vertices of $S_i$.  Consider the rooted tree $(\TT_1, S_0)$ formed by deleting $k_2$ free vertices of $S_2$ and adding $k_2$ free vertices to $S_1$.  Similarly, consider the rooted tree $(\TT_2, S_0)$ formed by deleting $k_1$ free vertices of $S_1$ and adding $k_1$ free vertices to $S_2$. Then $\ppc(\TT) \leq \max(\ppc(\TT_1), \ppc(\TT_2))$. \end{lemma}

Similar to branch shifting, the operation of simplex unbalancing preserves the quantity $N_k$ as defined in \ref{thebigguy}.  Therefore, this lemma allows us to rephrase proving Proposition \ref{thebigguy} for $\TT$ into proving Proposition \ref{thebigguy} for $\TT_1$ and $\TT_2$.  An example of simplex unbalancing is shown in Figure \ref{simplexunbalancingfig}, where we take $S_0$ to be the black triangle, $S_1$ to be the green triangle, and $S_2$ to be the blue triangle.

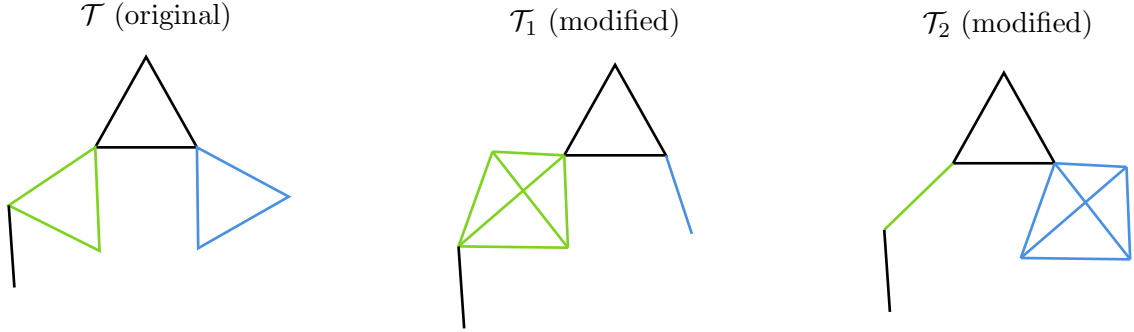
\begin{figure}[hbt!]
\centering

\tikzset{every picture/.style={line width=1pt}} 

\begin{tikzpicture}[x=0.75pt,y=0.75pt,yscale=-1,xscale=1]

\draw  [color={rgb, 255:red, 0; green, 0; blue, 0 }  ,draw opacity=1 ] (119.37,43) -- (144.74,88.51) -- (94,88.51) -- cycle ;
\draw  [color={rgb, 255:red, 74; green, 144; blue, 226 }  ,draw opacity=1 ] (190.57,113.31) -- (145.38,139.25) -- (144.74,88.51) -- cycle ;
\draw  [color={rgb, 255:red, 126; green, 211; blue, 33 }  ,draw opacity=1 ] (94,88.51) -- (96.14,140.57) -- (50.84,117.71) -- cycle ;
\draw    (50.84,117.71) -- (53.67,159) ;
\draw  [color={rgb, 255:red, 0; green, 0; blue, 0 }  ,draw opacity=1 ] (353.37,47) -- (378.74,92.51) -- (328,92.51) -- cycle ;
\draw [color={rgb, 255:red, 126; green, 211; blue, 33 }  ,draw opacity=0.5 ]   (292.21,90.67) -- (329.84,138.99) ;
\draw [color={rgb, 255:red, 74; green, 144; blue, 226 }  ,draw opacity=1 ]   (378.74,92.51) -- (391.67,132) ;
\draw [color={rgb, 255:red, 126; green, 211; blue, 33 }  ,draw opacity=1 ]   (292.21,90.67) -- (275.25,138.32) ;
\draw [color={rgb, 255:red, 126; green, 211; blue, 33 }  ,draw opacity=1 ]   (275.25,138.32) -- (329.84,138.99) ;
\draw [color={rgb, 255:red, 126; green, 211; blue, 33 }  ,draw opacity=1 ]   (329.84,138.99) -- (328,92.51) ;
\draw [color={rgb, 255:red, 126; green, 211; blue, 33 }  ,draw opacity=1 ]   (292.21,90.67) -- (328,92.51) ;
\draw [color={rgb, 255:red, 126; green, 211; blue, 33 }  ,draw opacity=0.5 ]   (275.25,138.32) -- (328,92.51) ;
\draw    (275.25,138.32) -- (278.08,179.6) ;
\draw  [color={rgb, 255:red, 0; green, 0; blue, 0 }  ,draw opacity=1 ] (547.37,51) -- (572.74,96.51) -- (522,96.51) -- cycle ;
\draw [color={rgb, 255:red, 126; green, 211; blue, 33 }  ,draw opacity=1 ]   (487.67,130) -- (522,96.51) ;
\draw    (487.67,130) -- (490.49,171.29) ;
\draw [color={rgb, 255:red, 74; green, 144; blue, 226 }  ,draw opacity=0.5 ]   (572.74,96.51) -- (610.38,144.84) ;
\draw [color={rgb, 255:red, 74; green, 144; blue, 226 }  ,draw opacity=1 ]   (572.74,96.51) -- (555.79,144.16) ;
\draw [color={rgb, 255:red, 74; green, 144; blue, 226 }  ,draw opacity=1 ]   (555.79,144.16) -- (610.38,144.84) ;
\draw [color={rgb, 255:red, 74; green, 144; blue, 226 }  ,draw opacity=1 ]   (572.74,96.51) -- (608.54,98.36) ;
\draw [color={rgb, 255:red, 74; green, 144; blue, 226 }  ,draw opacity=0.5 ]   (555.79,144.16) -- (608.54,98.36) ;
\draw [color={rgb, 255:red, 74; green, 144; blue, 226 }  ,draw opacity=1 ]   (608.54,98.36) -- (610.38,144.84) ;

\draw (85,14) node [anchor=north west][inner sep=0.75pt]   [align=left] {$\displaystyle \mathcal{T} \ $(original)};
\draw (299,17) node [anchor=north west][inner sep=0.75pt]   [align=left] {$\displaystyle \mathcal{T}_{1}$ (modified)};
\draw (505,18) node [anchor=north west][inner sep=0.75pt]   [align=left] {$\displaystyle \mathcal{T}_{2}$ (modified)};

\end{tikzpicture}

\caption{An example of simplex unbalancing.}
\label{simplexunbalancingfig}
\end{figure}

To prove simplex unbalancing, we need the following Lemma:

\begin{lemma} \label{step0lemma}
    Given a rooted simplex tree $(\TT, S_0)$ in $\mathbb{F}_q^d$, consider two distinct simplices $S_1$ and $S_2$, each of dimension at least 2 and with at least one free vertex.  Consider the rooted tree $(\TT', S_0)$ formed by deleting one free vertex of $S_1$ and adding one free vertex to $S_2$.  Then we have that
    \begin{align}
        \left( \frac{\Stab(\dim(S_1))}{\Stab(\dim(S_1'))} \cdot \frac{\Stab(\dim(S_2))}{\Stab(\dim(S_2'))} \right) q^{c(\TT ' )}\ =\ q^{c(\TT)}. \label{step0}
    \end{align} \bigskip
\end{lemma}

\begin{proof}
By Lemma \ref{numembeddingsofsimplextrees}, we see that $c(\TT) - c(\TT') = c(S_1) - c(S_1') + c(S_2) - c(S_2')$.  Let $\dim(S_1) = n$ and $\dim(S_2) = m$, so that $\dim(S_1') = n-1$ and $\dim(S_2') = m+1$.  We split into cases based on the dimension of simplices $S_1$ and $S_2$.  
\begin{itemize}
    \item If $n, m \geq d$, then by Lemma \ref{numberofconclassnsimp}, we see that $c(S)$ changes linearly by $d$ as the dimension of both simplices change by 1.  Therefore, $c(S_1) - c(S_1') + c(S_2) - c(S_2') = d + (-d) = 0$.  The sizes of the stabilizers are trivial in this case, so we have the equality in \eqref{step0}.
    
    \item If $n, m < d$, then by Lemma \ref{numberofconclassnsimp} we have that
    \begin{align}
        &c(S_1) - c(S_1') + c(S_2) - c(S_2') \nonumber \\ &=\ {\binom{n+1}  {2}} - {\binom{n} {2}} + {\binom{m+1}  {2}} - {\binom{m+2}  {2}}\ =\ n-m - 1.
    \end{align}
    Similarly, by Proposition \ref{stabsizelemma}, we see that
    \begin{equation}
        \frac{\Stab(n)}{\Stab(n-1)} \cdot \frac{\Stab(m)}{\Stab(m+1)}\ =\ q^{{\binom{d-n}{2}} - {\binom{d-n+1}{2}} + {\binom{d-m}{2}} - {\binom{d-m-1}  {2}}}\ =\ q^{n-m-1}.
    \end{equation}
    So again, we get the equality in \eqref{step0}.

    \item If $n < d, m \geq d$, then again by Lemma \ref{numberofconclassnsimp} we have that
    \begin{align}
        &c(S_1) - c(S_1') + c(S_2) - c(S_2') \nonumber \\ &=\ {\binom{n}  {2}} - {\binom{n-1} {2}} + d(m+1) - d(m+2)\ =\ n-d.
    \end{align}
    The stabilizers $\Stab(m)$ and $\Stab(m+1)$ are trivial, and we have by Proposition \ref{stabsizelemma} that
    \begin{equation}
        \frac{\Stab(n)}{\Stab(n-1)}\ =\ q^{{\binom{d-n}{2}} - {\binom{d-n+1}{2}}}\ =\ q^{n-d}.
    \end{equation}
    Thus, we have the equality on \eqref{step0}.  The case where $n \geq d, m < d$ follows similarly.
\end{itemize}
\end{proof}

\begin{proof}[Proof of Simplex Unbalancing]
For now, assume that $S_1$ and $S_2$ are not the root of $\TT$, and denote the parent vertices of $S_1$ and $S_2$ as $v_{0, 1}$ and $v_{0, 2}$ respectively.  Denote the set of free vertices of $S_1$ as $F_1$, and the set of free vertices of $S_2$ as $F_2$.  For $\ell \in \{1, 2\}$ the expression $R_{(\TT, S_0)}$ contains a term of the form

\begin{equation}
    D_{\mathcal{P}(S_\ell)}(u, u')\ =\ \frac{1}{\Stab(\dim(S_\ell))}\sum_{\theta \in \oq{d}}\lambda_{\theta}^{|F_\ell|}(u - \theta u') \prod_{\substack{v \in V(S_\ell) \\ v \neq v_{0, \ell} \\ v \notin F_\ell}}\sum_{\substack{x, x' \in E \\ x - \theta x' \\ = u - \theta u'}} \, \, \prod_{B \in \mathcal{B}(S_\ell, v)} D_{B}(x, x'). 
\end{equation}

Now consider the modified expression

\begin{align}
    &D_{\mathcal{P}(S_\ell)}(u, u', a_\ell)\ =& \notag \\ 
    &\hspace{0.5cm}\frac{1}{\Stab(\dim(S_\ell))}\sum_{\theta \in \oq{d}}\lambda_{\theta}^{a_\ell + (|F_\ell| - k_\ell)}(u - \theta u') \prod_{\substack{v \in V(S_\ell) \\ v \neq v_{0, \ell} \\ v \notin F_\ell}}\sum_{\substack{x, x' \in E \\ x - \theta x' \\ = u - \theta u'}} \, \, \prod_{B \in \mathcal{B}(S_\ell, v)}D_{B}(x, x') .
\end{align}

Consider the function $\psi(a_1, a_2)$, defined as the expression for $R_{(\TT, S_0)}$ where we substitute each instance of $D_{\mathcal{P}(S_{\ell})}(u, u')$ with $D_{\mathcal{P}(S_{\ell})}(u, u', a_\ell)$.  Therefore, we see that
\begin{equation}
R_{(\TT, S_0)}\ =\ \psi(k_1, k_2).
\end{equation}
It is clear that $\psi$ satisfies the conditions needed for the special case of Hadamard three-lines as stated in Corollary \ref{repeatedhadamard} (again by iterated triangle inequality), so applying this  gives us that

\begin{equation}
    \psi(k_1, k_2)\ \leq\ \psi(k_1 + k_2, 0)^{s_1}  \psi(0, k_1 + k_2)^{s_2}
\end{equation} for some positive $s_i$ such that $s_1 + s_2 = 1$.
Note that the only difference between $\psi(k_1 + k_2, 0)$ and $R_{(\TT_1, S_0)}$ is the stabilizer constant in front of $\psi(k_1 + k_2, 0)$'s modified $D_{\mathcal{P}(S_1)}$ and $D_{\mathcal{P}(S_2)}$ terms. This is because in the process of applying Hadamard three-lines above, we changed the dimension of $S_1$ and $S_2$ by changing the power on the $\lambda_{\theta}$ terms, but didn't change the stabilizer terms accordingly.  Therefore, we have that

\begin{equation}
    \psi(k_1 + k_2, 0)\ =\ \left( \frac{\Stab(\dim(S_1)+k_2)}{\Stab(\dim(S_1))} \cdot \frac{\Stab(\dim(S_2)-k_2)}{\Stab(\dim(S_2))} \right) R_{(\TT_1, S_0)}.
\end{equation}
By a repeated application of Lemma \ref{step0lemma}, we have that
\begin{equation}
    \left( \frac{\Stab(\dim(S_1)+k_2)}{\Stab(\dim(S_1))} \cdot \frac{\Stab(\dim(S_2)-k_2)}{\Stab(\dim(S_2))} \right) \frac{1}{q^{c(\TT_1)}}\ =\ \frac{1}{q^{c(\TT)}}.
\end{equation}
After noting that $|V(\TT_1)| = |V(\TT)|$, we see that for $s > \ppc(\TT_1, S_0)$,
\begin{align}
    \psi(k_1 + k_2, 0)\ &\lesssim\ \left( \frac{\Stab(\dim(S_1)+k_2)}{\Stab(\dim(S_1))} \cdot \frac{\Stab(\dim(S_2)-k_2)}{\Stab(\dim(S_2))} \right) \left(\eoq{2|V(\TT_1)|}{c(\TT_1)}\right) \\ 
    &\lesssim\ \eoq{2|V(\TT)|}{c(\TT)}.
\end{align} 
By an analogous argument, we also have that for $s > \ppc(\TT_2, S_0)$, 
\begin{align}
    \psi(0, k_1 + k_2)\ &\lesssim \ \left( \frac{\Stab(\dim(S_1)-k_1)}{\Stab(\dim(S_1))} \cdot \frac{\Stab(\dim(S_2)+k_1)}{\Stab(\dim(S_2))} \right) \left(\eoq{2|V(\TT_2)|}{c(\TT_2)}\right) \\ 
    &\lesssim\ \eoq{2|V(\TT)|}{c(\TT)}.
\end{align} 

Therefore, whenever $s > \max(\ppc(\TT_1),\ppc(\TT_2)) $, we may conclude that $R_{(\TT, S_0)} = \psi(k_1, k_2) \lesssim E^{2|V(\TT)|}q^{-c(\TT)}$. \bigskip

In the case that $S_1$ or $S_2$ is the root of $(\TT, S_0)$, the proof follows similarly.  Without loss of generality, let $S_1$ be the root of $\TT$.  Again, let $F_1$ and $F_2$ be the set of free vertices of $S_1$ and $S_2$ respectively.  The term $R_{(\TT, S_0)}$ is of the form

\begin{equation}
     R_{(\TT, S_0)}\ =\ \sum_{w \in \mathbb{F}_q^d} \, \sum_{\theta \in \oq{d}} \lambda_{\theta}^{|F_1|}(w) \prod_{\substack{v \in V(S_1) \\ v \notin F_1}} \, \sum_{\substack{x, x' \in E \\ x - \theta x' \\ = w}} \, \prod_{B \in \mathcal{B}(S_1, v)}D_{B}(x, x') 
\end{equation}
Define $\psi(a_1, a_2)$ as the modified expression
\begin{equation}
    \sum_{w \in \mathbb{F}_q^d} \, \sum_{\theta \in \oq{d}} \lambda_{\theta}^{a_1 + |F_1| - k_1}(w) \prod_{\substack{v \in V(S_1) \\ v \notin F_1}} \sum_{\substack{x, x' \in E \\ x - \theta x' \\ = w}} \prod_{B \in \mathcal{B}(S_1, v)}D_{B}(x, x') ,
\end{equation}
where we substitute each instance of $D_{\mathcal{P}(S_2)}(u, u')$ with $D_{\mathcal{P}(S_2)}(u, u', a_\ell)$ as we did in the previous case.  Applying Hadamard three-lines and arguing as in the previous case completes the proof of the Lemma.
\end{proof}

\subsection{Proof of Proposition \ref{thebigguy} and Theorem \ref{extendsmall2023}}


We have the following simple fact that we use throughout this subsection.
\begin{lemma}\label{edgeaddition2}
    Consider a graph $G$ and a subgraph $H$ of $G$.  If for some $s \in \mathbb{R}$, $E \subset \mathbb{F}_q^d$ contains a positive proportion of congruence classes of embeddings of $G$ in $\mathbb{F}_q^d$ whenever $|E| \gtrsim q^s$, then $E$ contains a positive proportion of congruence classes of embeddings of $H$ in $\mathbb{F}_q^d$ whenever $|E| \gtrsim q^s$.
\end{lemma}
\begin{proof}
    Note that a congruence class of $G$ determines a congruence class of $H$.  As long as the number of congruence classes of $H$ in $\FF_q^d$ is unbounded as $q$ grows, it is clear that $E$ cannot contain a positive proportion of congruence classes of embeddings of $G$ without containing a positive proportion of congruence classes of embeddings of $H$. \bigskip
\end{proof}

\begin{proof}[Proof of Theorem \ref{extendsmall2023}]
We may assume that all simplices in $\TT$ are exactly dimension $k$ by Lemma \ref{edgeaddition2}.  We first choose a root $S_0$ of $\TT$.  Then, to any simplex $S \in \TT$ such that $S$ has more than one child vertex, we apply branch shifting to two child vertices of $S$ to reduce to to a collection of modified $\TT_i$ where $S$ has one less child vertex.  Repeating this process for all $S \in \TT$ allows us to reduce to the case where $\TT$ is a simplex tree consisting only of $n$-simplices with at most one child vertex.  Simplex trees of this form are $k$-weak simplex trees, and we apply Theorem 1.12 of \cite{small2023} to conclude that for $s = k + \frac{d-1}{2}$, $|E| \gtrsim q^s$ implies that $E$ contains all congruence classes of $\TT$ in $\fqd$. \bigskip
\end{proof}

We first prove a special case of Proposition \ref{thebigguy} for a simple class of simplex trees $\TT$ before continuing to the proof of \ref{thebigguy} in full generality.

\begin{lemma}\label{simplebigguy1}
    For an integer $k$ such that $1 \leq k < \frac{d+1}{2}$, let $\TT$ be a simplex tree consisting of exactly one simplex $S$ of dimension greater than $k$.  For $s = \max \left(\frac{d\dim(S) + 1}{\dim(S)+1}, \, \, k + \frac{d-1}{2} \right) $, we have that for $E \subset \fqd$, whenever $|E| \gtrsim q^s$,
    \begin{equation}
         \stsum \ \lesssim \ \eoq{2|V(\TT)|}{c(\TT)}.
    \end{equation}
    If we additionally impose the condition that $d=2$, $k=1$ and that $q \equiv 3 \pmod 4$, for $s = \frac{4\dim(S)}{2\dim(S) + 1}$ we have that whenever $|E| \gtrsim q^s$ the above holds as well.
\end{lemma} 

\begin{proof}
We view $S$ as the root of $\TT$, and by Lemma \ref{edgeaddition2}, assume that all simplices of dimension at most $k$ are exactly dimension $k$.  If $S$ has no child vertices, we directly apply the bounds found in \cite{groupactions} and \cite{alexmcdonald} and no further work is needed.  If $S$ has at least one child vertex, we may reduce to the case where $S$ has exactly one child vertex.  To do this, we apply branch shifting to a pair of child vertices of $S$ to reduce to the case where $S$ has one less child vertex, and repeat this process.  Denote the sole child vertex of $S$ as $v$. \bigskip

Next, we modify the branches of $S$.  To any $k$-simplex with multiple child vertices, we repeatedly apply branch shifting to reduce to the case that each $k$-simplex only has a single child vertex.  After this modification, each branch of $S$ is of the form of a $k$-weak simplex tree as defined in Section 3.  Then we may write the sum $R_{(\TT, S)}$ as follows:

\begin{equation}
    R_{(\TT, S)}\ =\  \frac{1}{\Stab(\dim(S))} \sum_{w \in \mathbb{F}_q^d} \, \, \sum_{\theta \in \oq{d}} \, \,  \lambda_{\theta}^{\dim(S)}(w)  \, \, \sum_{\substack{x, x' \in E \\ x - \theta x' = w}} \, \, \prod_{B \in \mathcal{B}(S, v)}D_{B}(x, x').
\end{equation}
We note that $\prod_{B \in \mathcal{B}(S, v)}D_{B}(x, x')$ counts the number of congruent embeddings (based at $x$ and $x'$) of the structure formed by the union of all the branches of $S$ at $v$.  We note that the union of all of the branches of $S$ at $v$ is a $k$-weak simplex tree $T_k$ rooted at $v$.  Therefore, $\prod_{B \in \mathcal{B}(S, v)}D_{B}(x, x')$ counts the number of congruent pairs of embeddings of $T_k$ into $E$ based at $x$ and $x'$. \bigskip

Defining $e(T_k) = \ell\binom{k+1}{2}$ as the number of edges of $T_k$ as in Section 3, let $t_1, \ldots, t_{e(T)} \in \mathbb{F}_q$  determine a congruence class of embeddings of $T_k$.  Defining $f_{T_k, t_1, \ldots, t_{e(T_k)}}$ as in Section 3, we may write

\begin{align}
    &R_{(\TT, S)} \ = \ \frac{1}{\Stab(\dim(S))}\sum_{\theta, w} \lambda_{\theta}^{\dim(S)}(w) \sum_{\substack{x, x' \in E \\ x - \theta x' = w}} \, \ \prod_{B \in \mathcal{B}(S, v)}D_{B}(x, x') \nonumber \\ 
    &=\ \frac{1}{\Stab(\dim(S))}\sum_{t_1, \ldots, t_{e(T_k)}} \sum_{\theta, w} \lambda_{\theta}^{\dim(S)}(w) \sum_{\substack{x, x' \in E \\ x - \theta x' = w}} f_{T_k, t_1, \ldots, t_{e(T_k)}}(x) f_{T_k, t_1, \ldots, t_{e(T_k)}}(x'). \label{wherewesplit}
\end{align}

We also may assume that $t_1, \ldots, t_{e(T_k)}$ range over nonzero elements in $\mathbb{F}_q$ rather than all of $\mathbb{F}_q$.  This is because we only care about containing a positive proportion of congruence classes of $\TT$ in $\fqd$, and the number of congruence classes of $T_k$ determined by a set of $e(T_k)$ distances with at least one distance zero is negligible compared to the number of congruence classes determined by a set of $e(T_k)$ nonzero distances.  Fixing a congruence class $t_1, \ldots, t_{e(T_k)}$ of $T_k$, we note that the inside of the outermost sum of \eqref{wherewesplit} is exactly of the form $\sum_{\theta, w} \lambda_{\theta}^{\dim(S)} (w) \Gamma_{\theta, T_k}(w)$, where $\Gamma_{\theta, T_k}(w)$ is defined as in Section 3.  From this point onward, we split into cases for $\mathbb{F}_q^2$ and $\mathbb{F}_q^d$.  For notational convenience, we let $n$ denote the dimension of $S$ for the rest of this Lemma.  \bigskip

In $\mathbb{F}_q^2$, from an immediate application of Lemma \ref{hammerin2d} we have that for $|E| \geqsim q^{\frac{4n}{2n+1}}$,
\begin{align}
    &\frac{1}{\Stab(n)}\sum_{t_1, \ldots, t_{e(T_k)}} \sum_{\theta, w} \lambda_{\theta}^{n}(w) \sum_{\substack{x, x' \in E \\ x - \theta x' = w}} f_{T_k, t_1, \ldots, t_{e(T_k)}}(x) f_{T_k, t_1, \ldots, t_{e(T_k)}}(x') \nonumber \\
    &\lesssim \ \frac{1}{\Stab(n)}q^{\ell\binom{k+1}{2}} \left(\eoq{2\ell k+2n+2}{2\ell\binom{k+1}{2} + 2n-1}\right) \ = \ \frac{1}{\Stab(n)}\left(\eoq{2\ell k+2n+2}{\ell\binom{k+1}{2} + 2n-1}\right).
\end{align}

It is easy to show that $\TT$ has $\ell k + n + 1$ vertices.  We also see that by Lemma \ref{numembeddingsofsimplextrees}, $c(\TT) = c(S) + \ell\binom{k+1}{2}$.  In $d = 2$, we see that $\Stab(n) \approx 1$ and $c(S) = 2n-1$ by Lemma \ref{numberofconclassnsimp}.  Therefore, we have shown that for $|E| \geqsim q^{\frac{4n}{2n+1}}$,
\begin{equation}
    \sum_{d \in \mathbb{D}} \nu^2_{\TT}(d) \ \approx \ R_{(\TT, S)} \ \lesssim \ {\eoq{2|V(\TT)|}{c(\TT)}}.
\end{equation}

In $\mathbb{F}_q^d$, $d > 2$, we apply Lemma \ref{hammerindd} to \eqref{wherewesplit}, so for $|E| \gtrsim q^{\max(\frac{dn+1}{n+1}, \, k + \frac{d-1}{2})}$,

\begin{align}
    &\frac{1}{\Stab(n)}\sum_{t_1, \ldots, t_{e(T_k)}} \sum_{\theta, w} \lambda_{\theta}^{n}(w) \sum_{\substack{x, x' \in E \\ x - \theta x' = w}} f_{T_k, t_1, \ldots, t_{e(T_k)}}(x) f_{T_k, t_1, \ldots, t_{e(T_k)}}(x') \\
    &\lesssim\ \frac{1}{\Stab(n)}q^{\ell\binom{k+1}{2}} \left(\eoq{2\ell k + 2n +2}{2\ell\binom{k+1}{2} + d(n+1) - \binom{d+1}{2}}\right) \\ 
    &=\  \frac{1}{\Stab(n)}\left(\eoq{2\ell k + 2n +2}{\ell\binom{k+1}{2} + d(n+1) - \binom{d+1}{2}}\right).
\end{align}

Again we see that $\TT$ has $\ell k + n + 1$ vertices, and that $c(\TT) = c(S) + l\binom{k+1}{2}$.  When $\dim(S) = n \geq d-1$, we have that $\Stab(n) \approx 1$ and that $c(S) = d(n+1) - \binom{d+1}{2}$.  When $\dim(S) = n < d-1$, we have that $\Stab(n)$ is $q^{\binom{d-n}{2}}$ by Proposition \ref{stabsizelemma}.  As $\binom{d-n}{2} + d(n+1) - \binom{d+1}{2} = \binom{n+1}{2}$, we see that for $|E| \gtrsim q^{\max(\frac{dn+1}{n+1}, \, k + \frac{d-1}{2})}$,
\begin{equation}
    \sum_{d \in \mathbb{D}} \nu^2_{\TT}(d) \ \approx \ R_{(\TT, S)} \ \lesssim \ {\eoq{2|V(\TT)|}{c(\TT)}}.
\end{equation}
\end{proof}

\begin{proof} [Proof of Proposition \ref{thebigguy} for general simplex trees]
    
Fix an integer $k$ in the statement of Proposition \ref{thebigguy} such that $1 \leq k < \frac{d+1}{2}$.  If there are less than two simplices of dimension greater than $k$ in $\TT$, then we are done by Lemma \ref{simplebigguy1} and Theorem \ref{extendsmall2023}.  Now suppose there are at least two simplices in $\TT$ of dimension greater than $k$.  Then there must be at least two $k$-leafs in $\mathcal{T}$.  Fix two distinct $k$-leafs $S_1$ and $S_2$ of dimensions $n, m > k$ respectively.
\begin{itemize}
    \item Step 1: We reduce to the case where one of $S_1, S_2$ has at most one child vertex and the other has no child vertices.  To do this, we iteratively apply branch shifting.  Consider the case where $S_1$ has more than one child vertex.  Here, we apply branch shifting to a pair of child vertices of $S_1$ to reduce to the case where $S_1$ has one less child vertex, and repeat this process until $S_1$ has a single child vertex.  Therefore, we may assume that $S_1$ has either zero or one child vertices.  Similarly, we may assume that $S_2$ has either zero or one child vertices.  If either $S_1$ or $S_2$ has no child vertices, we are done.  Otherwise, apply branch shifting to the child vertex of $S_1$ and the child vertex of $S_2$.

    \item Step 2: Denote $F_1$ and $F_2$ as the set of free vertices of $S_1$ and $S_2$.  If $S_1$ has a child vertex, set $k_1 = |F_1| - k + 1$, otherwise, set $k_1 = |F_1| - k$.  Similarly, if $S_2$ has a child vertex, set $k_2 = |F_2| - k + 1$, otherwise, set $k_2 = |F_2| - k$.  Then apply simplex unbalancing to $S_1$ and $S_2$ with the appropriate $k_1$, $k_2$.  Note that in all of these cases, this process lets us reduce to the case where one of $S_1$, $S_2$ is an $k$-simplex, and the other is a $(\dim(S_1) + \dim(S_2) - k)$-dimensional simplex.  
\end{itemize}
After an application of these two steps, we have reduced the goal of proving Proposition \ref{thebigguy} for $\TT$ to proving Proposition \ref{thebigguy} for a collection of $\TT_i$, all with one less simplex of dimension greater than $k$ than $\TT$, since the quantity $N_k$ is invariant under branch shifting and simplex unbalancing.  Note that for each of these $\TT_i$, we have that 
\begin{equation}
    \sum_{\substack{S \in \TT_i \\ \dim(S) > k}} \dim(S) \ = \ \left(\sum_{\substack{S \in \TT \\ \dim(S) > k}} \dim(S)\right) - k.
\end{equation}

Therefore, if we start with a simplex tree $\TT$, and repeatedly apply the above two steps to reduce proving Proposition \ref{thebigguy} for $\TT$ to proving Proposition \ref{thebigguy} for a collection of $\TT_j$ each containing exactly one simplex $S_j$ of dimension greater than $k$, we see that for each $\TT_j$,
\begin{equation}
    \dim(S_j) \ = \ k+ \sum_{\substack{S \in \TT \\ \dim(S) > k}} (\dim(S) - k)
\end{equation}

Applying Lemma \ref{simplebigguy1} to each $T_j$ completes the proof of Proposition \ref{thebigguy} in full generality.

\end{proof}

\section{Cycles of simplices}

In this section, we discuss partial results related to cycles of simplices.

\begin{definition}
    Let $\mathcal{C}$ be a finite set of simplices $S_1, \ldots S_k$ where $S_i$ shares a vertex with $S_{i+1}$ for all $1 \leq i \leq k-1$, and $S_k$ shares a vertex with $S_1$.  Then we call $\CC$ a cycle of simplices.
\end{definition}

Our goal is to find an $s$ such that $\mathbb{F}_q^d$ contains a positive proportion of congruence classes of $\CC$ in $\mathbb{F}_q^d$ whenever $|E| \gtrsim q^s$.  As argued for simplex trees in Lemma \ref{numembeddingsofsimplextrees}, we have the following:

\begin{lemma}
    The number of congruence classes of $\CC$ in $\mathbb{F}_q^d$ is on the order of $q^{\sum_{S \in \CC} c(S)}$.  We denote $\sum_{S \in \CC} c(S)$ as $c(\CC)$.
\end{lemma}

Following the framework developed in Section 4, to find an $s$ corresponding to a given simplex cycle $\CC$, we first reduce this goal finding an $s$ for a simple class of simplex cycles consisting of one large simplex and all other simplices of dimension one. This is more precisely described in the proposition below.

\begin{prop} \label{goalof5.1}
    Consider a cycle of simplices $\mathcal{C}$ with $k$ simplices, and define.
    $$N\ =\ 1 + \sum_{S \in \CC, \dim(S) > 1} (\dim(S) - 1)$$
    Let $\mathcal{C}'$ be a cycle of simplices containing $k-1$ simplices of dimension 1 and one simplex of dimension $N$.  Then if for some $s \in \mathbb{R}$, $E$ contains a positive proportion of congruence classes of $\CC'$ in $\mathbb{F}_q^d$ whenever $|E| \gtrsim q^s$, then the same is true for $\CC$ and the same $s$ as well.
\end{prop}

The proof of the above Proposition requires a modification of the simplex unbalancing technique from earlier.  However after we reduce to this class of simple cycles, we find that various technical obstructions occur in trying to generalize the methods of Section 3 to this setting.  Morally, however, we expect the $s$ for $\CC'$ to be slightly larger or equal to that of simplices of dimension $N$, as cycles of arbitrary length are much more abundant than simplices of dimension $N$ in $\mathbb{F}_q^d$.




\subsection{The setup and Hadamard three-lines reduction}

Consider a cycle of simplices $\mathcal{C}$ given by a collection of simplices $S_1, \ldots, S_k$, where $S_i$ shares a vertex with $S_{i+1}$ for all $1 \leq i \leq k-1$ and $S_k$ shares a vertex with $S_1$.  Again, we are interested in finding an $s$ such that $\mathbb{F}_q^d$ contains a positive proportion of congruence classes of $\CC$ in $\mathbb{F}_q^d$ whenever $|E| \gtrsim q^s$.  Using the same Cauchy-Schwarz setup as in Section 4.1, we show that this problem is equivalent to finding an $s$ such that whenever $|E| \gtrsim q^s$,
\begin{equation}
    \sum_{\delta \in \mathbb{D}} \nu_{\mathcal{C}}^2(\delta) \ \lesssim\ \eoq{2|V(\mathcal{C})|}{c(\mathcal{C})}
\end{equation}
As usual, $\mathbb{D}$ denotes the space of congruence classes of $\mathcal{C}$, $\nu_{\mathcal{C}}(\delta)$ denotes the number of embeddings in the congruence class $\delta$ contained in $E$, and $V(\mathcal{C})$ denotes the vertex set of $\mathcal{C}$.  Analogous to simplex trees, a vertex in $\mathcal{C}$ is called a free vertex if it is shared by only one simplex.  As per Section 4.1, the goal is now be to write $\stsumc$ in a way that is more amenable to applying Hadamard three-lines.  We have two different ways of doing this, depending on whether the two simplices $S_1, S_2 \in \CC$ we distinguish share a vertex or not. \bigskip

First, consider nonadjacent simplices $S_1$, $S_2$ of $\mathcal{C}$ with dimension $m$ and $n$ respectively. Denote $\mathcal{H}_1$ and $\mathcal{H}_2$ as the two chains of simplices connecting $S_1$ to $S_2$, denote $v_{1, i}$ as the vertex shared by $S_1$ and $\mathcal{H}_i$, and denote $v_{2, i}$ similarly as the vertex shared by $S_2$ and $\mathcal{H}_i$.  Let $\delta_1$ and $\delta_2$ be respective congruence class of these two chains.  Let the function $f_{\delta_1}(x_1, x_2)$ count the number of embeddings of $\mathcal{H}_1$ into $E$ in a given congruence class $\delta_1$, such that the embedding sends $v_{1, 1}$ to $x_1$ and $v_{1, 2}$ to $x_2$.  Similarly, let the function $g_{\delta_2}(x_1, x_2)$ count the number of embeddings of $\mathcal{H}_2$ into $E$ in a given congruence class $\delta_2$, such that the embedding sends $v_{2, 1}$ to $x_1$ and $v_{2, 2}$ to $x_2$.  Then we define
\begin{align}
    D_{\text{nonadj}}(\CC, S_1, S_2)\ &=\ \frac{1}{\Stab(m) \Stab(n)} \sum_{\delta_1, \delta_2} \, \, \sum_{\theta, \phi \in \oq{d}} \, \, \sum_{w_1, w_2 \in \mathbb{F}_q^d} \lambda_{\theta}^{m-1}(w_1)\lambda_{\phi}^{n-1}(w_2) \notag \\ 
    &\hspace{0.5cm}\sum_{\substack{u_1, u_1', v_1, v_1' \in E \\ u_1 - \theta u_1' = \\ v_1 - \theta v_1' = w_1}} \, \, \sum_{\substack{u_2, u_2', v_2, v_2' \in E \\ u_2 - \phi u_2' = \\ v_2 - \phi v_2' = w_2}} f_{\delta_1}(u_1, u_2) f_{\delta_1}(u_1', u_2') g_{\delta_2}(v_1, v_2) g_{\delta_2}(v_1', v_2').
\end{align}
\begin{prop}
For nonadjacent simplices $S_1, S_2 \in \CC$ with dimensions $m$ and $n$ respectively, we have that
    \begin{equation}
        \stsumc\ \approx\ D_{\mathrm{nonadj}}(\CC, S_1, S_2).
    \end{equation}
\end{prop}
\begin{proof}
    We first count the number of pairs of congruent embeddings $h, h': V(S_1 \cup S_2) \to E$ of the collection of two disjoint simplices $S_1, S_2$.  Let the dimensions of $S_1$ and $S_2$ be $m$ and $n$ respectively. By the methods of Section 4.1, we see that this is on the order of
    \begin{align}
        &\frac{1}{\Stab(m) \Stab(n)}  \sum_{\theta, \phi \in \oq{d}} \, \, \sum_{w_1, w_2 \in \mathbb{F}_q^d} \lambda_{\theta}^{m+1}(w_1)\lambda_{\theta}^{n+1}(w_2) \\
        &=\ \frac{1}{\Stab(m) \Stab(n)}  \sum_{\theta, \phi \in \oq{d}} \, \, \sum_{w_1, w_2 \in \mathbb{F}_q^d} \lambda_{\theta}^{m-1}(w_1)\lambda_{\phi}^{n-1}(w_2) \sum_{\substack{u_1, u_1', v_1, v_1' \in E \\ u_1 - \theta u_1' = \\ v_1 - \theta v_1' = w_1}} \, \, \sum_{\substack{u_2, u_2', v_2, v_2' \in E \\ u_2 - \phi u_2' = \\ v_2 - \phi v_2' = w_2}} 1.
    \end{align}
In the equality above, we just unpack the definition of $\lambda_{\theta}$ and $\lambda_{\phi}$.  For every congruent pair of embeddings $h, h'$ the expression above counts, $u_1$ and $v_1$ are the vertices $h(v_{1, 1})$ and $h(v_{1, 2})$ of $h(S_1)$, and $u_1'$ and $v_1'$ are the vertices $h'(v_{1, 1})$ and $h'(v_{1, 2})$ of $h'(S_1)$.  Similarly $u_2$ and $v_2$ are the the vertices $h(v_{2, 1})$ and $h(v_{2, 2})$ of $h(S_2)$, and $u_2'$ and $v_2'$ are the vertices $h'(v_{2, 1})$ and $h'(v_{2, 2})$ of $h'(S_2)$.  See Figure \ref{cyclepicnonadj} for an example where $\dim(S_1) = 2$ and $\dim(S_2) = 3$. \bigskip

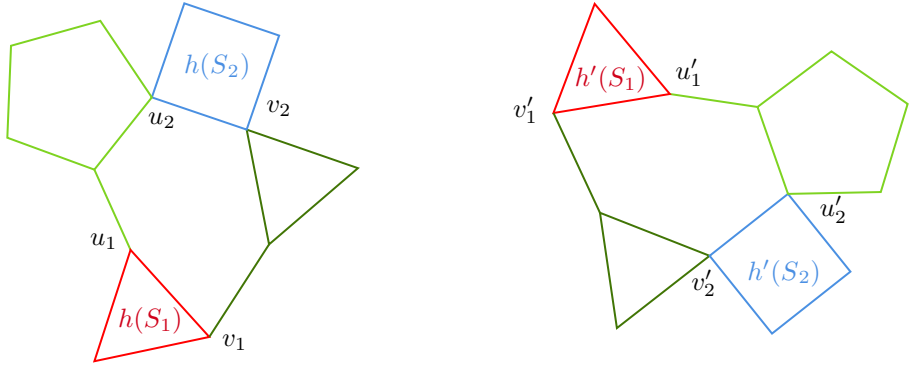
\begin{figure}[hbt!]
\centering

\tikzset{every picture/.style={line width=0.75pt}} 

\begin{tikzpicture}[x=0.75pt,y=0.75pt,yscale=-1,xscale=1]

\draw  [color={rgb, 255:red, 65; green, 117; blue, 5 }  ,draw opacity=1 ] (216.6,114.46) -- (272.12,133.86) -- (227.78,172.21) -- cycle ;
\draw  [color={rgb, 255:red, 74; green, 144; blue, 226 }  ,draw opacity=1 ] (169.29,98.25) -- (185.5,50.95) -- (232.8,67.16) -- (216.6,114.46) -- cycle ;
\draw  [color={rgb, 255:red, 126; green, 211; blue, 33 }  ,draw opacity=1 ] (169.29,98.25) -- (140.61,134.6) -- (97.19,118.55) -- (99.03,72.29) -- (143.59,59.75) -- cycle ;
\draw [color={rgb, 255:red, 126; green, 211; blue, 33 }  ,draw opacity=1 ]   (140.61,134.6) -- (158.67,175) ;
\draw  [color={rgb, 255:red, 255; green, 0; blue, 0 }  ,draw opacity=1 ] (158.67,175) -- (197.96,218.77) -- (140.63,230.99) -- cycle ;
\draw [color={rgb, 255:red, 65; green, 117; blue, 5 }  ,draw opacity=1 ]   (227.78,172.21) -- (197.96,218.77) ;
\draw  [color={rgb, 255:red, 65; green, 117; blue, 5 }  ,draw opacity=1 ] (447.58,177.95) -- (401.33,214.31) -- (392.88,156.3) -- cycle ;
\draw  [color={rgb, 255:red, 74; green, 144; blue, 226 }  ,draw opacity=1 ] (486.69,146.81) -- (517.84,185.92) -- (478.72,217.07) -- (447.58,177.95) -- cycle ;
\draw  [color={rgb, 255:red, 126; green, 211; blue, 33 }  ,draw opacity=1 ] (486.69,146.81) -- (471.49,103.08) -- (508.39,75.11) -- (546.39,101.56) -- (532.98,145.87) -- cycle ;
\draw [color={rgb, 255:red, 126; green, 211; blue, 33 }  ,draw opacity=1 ]   (471.49,103.08) -- (427.72,96.55) ;
\draw  [color={rgb, 255:red, 255; green, 0; blue, 0 }  ,draw opacity=1 ] (427.72,96.55) -- (369.68,106.11) -- (390.24,51.21) -- cycle ;
\draw [color={rgb, 255:red, 65; green, 117; blue, 5 }  ,draw opacity=1 ]   (392.88,156.3) -- (369.68,106.11) ;

\draw (137,165.4) node [anchor=north west][inner sep=0.75pt]  [font=\small]  {$u_{1}$};
\draw (202.96,216.17) node [anchor=north west][inner sep=0.75pt]  [font=\small]  {$v_{1}$};
\draw (166,104.4) node [anchor=north west][inner sep=0.75pt]  [font=\small]  {$u_{2}$};
\draw (225,98.4) node [anchor=north west][inner sep=0.75pt]  [font=\small]  {$v_{2}$};
\draw (429,77.4) node [anchor=north west][inner sep=0.75pt]  [font=\small]  {$u'_{1}$};
\draw (501,146.4) node [anchor=north west][inner sep=0.75pt]  [font=\small]  {$u'_{2}$};
\draw (437,183.4) node [anchor=north west][inner sep=0.75pt]  [font=\small]  {$v'_{2}$};
\draw (348,96.4) node [anchor=north west][inner sep=0.75pt]  [font=\small]  {$v'_{1}$};
\draw (151,202.4) node [anchor=north west][inner sep=0.75pt]  [font=\small,color={rgb, 255:red, 208; green, 2; blue, 27 }  ,opacity=1 ]  {$h( S_{1})$};
\draw (378.84,81) node [anchor=north west][inner sep=0.75pt]  [font=\small,color={rgb, 255:red, 208; green, 2; blue, 27 }  ,opacity=1 ]  {$h'( S_{1})$};
\draw (184,74.4) node [anchor=north west][inner sep=0.75pt]  [font=\small,color={rgb, 255:red, 208; green, 2; blue, 27 }  ,opacity=1 ]  {$\textcolor[rgb]{0.29,0.56,0.89}{h}\textcolor[rgb]{0.29,0.56,0.89}{(}\textcolor[rgb]{0.29,0.56,0.89}{S}\textcolor[rgb]{0.29,0.56,0.89}{_{2}}\textcolor[rgb]{0.29,0.56,0.89}{)}$};
\draw (465,178.4) node [anchor=north west][inner sep=0.75pt]  [font=\small,color={rgb, 255:red, 208; green, 2; blue, 27 }  ,opacity=1 ]  {$\textcolor[rgb]{0.29,0.56,0.89}{h'}\textcolor[rgb]{0.29,0.56,0.89}{(}\textcolor[rgb]{0.29,0.56,0.89}{S}\textcolor[rgb]{0.29,0.56,0.89}{_{2}}\textcolor[rgb]{0.29,0.56,0.89}{)}$};

\end{tikzpicture}
\caption{A potential extension $\displaystyle \tilde{h}$ and $\displaystyle \widetilde{h'}$ of two embeddings $\displaystyle h,\ h':\ V( S_{1} \cup S_{2}) \ \rightarrow E$. }
\label{cyclepicnonadj}
\end{figure}

We now weight the count of each pair of congruent embeddings $h, h': V(S_1 \cup S_2) \to E$ by how many pairs of congruent embeddings $\tilde{h}, \tilde{h'}: V(\CC) \to E$ it extends to.  This is exactly

\begin{equation}
    \sum_{\delta_1, \delta_2} f_{\delta_1}(u_1, u_2) f_{\delta_1}(u_1', u_2') g_{\delta_2}(v_1, v_2) g_{\delta_2}(v_1', v_2'),
\end{equation}
completing the proof of the proposition. \bigskip
\end{proof}

Now consider adjacent simplices $S_1$, $S_2$ of $\CC$ with dimensions $m$ and $n$ respectively.  Denote $\mathcal{H}$ as the chain of simplices connecting $S_1$ to $S_2$, denote $v_{1}$ as the vertex shared by $\mathcal{H}$ and $S_1$, and denote $v_{2}$ similarly as the vertex shared by $\mathcal{H}$ and $S_2$.  Let $\delta'$ be a congruence class of $\mathcal{H}$.  Let the function $F_{\delta'}(x_1, x_2)$ count the number of embeddings of $\mathcal{H}$ into $E$ in the given congruence class $\delta'$, such that the embedding sends $v_{1}$ to $x_1$ and $v_{2}$ to $x_2$.  Then we define
\begin{align}
    D_{\text{adj}}(\CC, S_1, S_2)\ =\ \frac{1}{\Stab(m) \Stab(n)} \sum_{\delta'} \, \, \sum_{\theta, \phi \in \oq{d}} \, \, \sum_{x_1, x_2 \in \mathbb{F}_q^d} \lambda_{\theta}^{m-1}(x - \theta x')\lambda_{\phi
    }^{n-1}(x-\phi x') \notag \\ \sum_{\substack{u_1, u_1' \in E \\ u_1 - \theta u_1' \\ = x - \theta x'}} \, \, \sum_{\substack{u_2, u_2' \in E \\ u_2 - \phi u_2' \\ = x - \phi x'}} F_{\delta'}(u_1, u_2) F_{\delta'}(u_1', u_2').
\end{align}

\begin{prop}
For adjacent simplices $S_1, S_2 \in \CC$ with dimensions $m$ and $n$ respectively, we have that
    \begin{equation}
        \stsumc\ \approx\ D_{\mathrm{adj}}(\CC, S_1, S_2).
    \end{equation}
\end{prop}
\begin{proof}
We first count the number of pairs $h, h': V(S_1 \cup S_2) \to E$ of congruent embeddings of the collection of two simplices $S_1, S_2$ sharing a vertex.  Then by the methods of Section 4.1, we see that this is on the order of
\begin{align}
    &\frac{1}{\Stab(m) \Stab(n)}\sum_{\theta, \phi \in \oq{d}} \, \, \sum_{x_1, x_2 \in \mathbb{F}_q^d} \lambda_{\theta}^{m}(x - \theta x')\lambda_{\phi
    }^{n}(x-\phi x') \\
    &=\ \frac{1}{\Stab(m) \Stab(n)}\sum_{\theta, \phi \in \oq{d}} \, \, \sum_{x_1, x_2 \in \mathbb{F}_q^d} \lambda_{\theta}^{m-2}(x - \theta x')\lambda_{\phi
    }^{n-2}(x-\phi x')\sum_{\substack{u_1, u_1' \in E \\ u_1 - \theta u_1' \\ = x - \theta x'}} \, \, \sum_{\substack{u_2, u_2' \in E \\ u_2 - \phi u_2' \\ = x - \phi x'}} 1.
\end{align}
For every congruent pair of embeddings $h, h'$ the expression above counts, $u_1$ and $u_2$ are the vertices $h(v_{1})$ and $h(v_2)$, and $u_1'$ and $u_2'$ are the vertices $h'(v_{1})$ and $h'(v_2)$. See Figure \ref{cyclespic2} for an example where $\dim(S_1) = 2$ and $\dim(S_2) = 3$. \bigskip

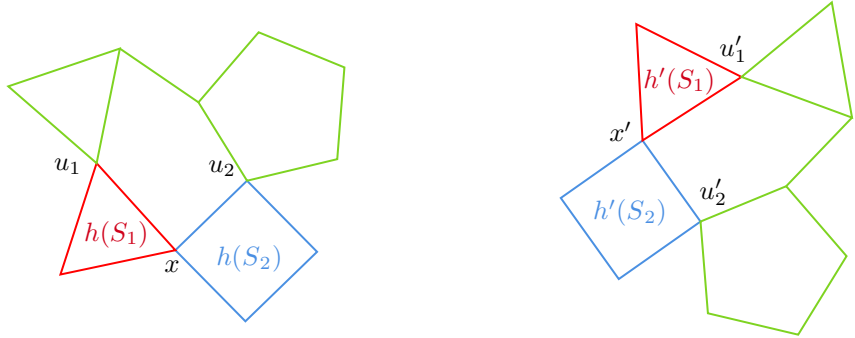
\begin{figure}[hbt!]
\centering

\tikzset{every picture/.style={line width=0.75pt}} 

\begin{tikzpicture}[x=0.75pt,y=0.75pt,yscale=-1,xscale=1]

\draw  [color={rgb, 255:red, 74; green, 144; blue, 226 }  ,draw opacity=1 ] (185.96,129.44) -- (221.72,94.49) -- (256.67,130.25) -- (220.91,165.2) -- cycle ;
\draw  [color={rgb, 255:red, 255; green, 0; blue, 0 }  ,draw opacity=1 ] (146.67,85.67) -- (185.96,129.44) -- (128.63,141.66) -- cycle ;
\draw  [color={rgb, 255:red, 126; green, 211; blue, 33 }  ,draw opacity=1 ] (221.72,94.49) -- (197.49,55.04) -- (227.52,19.81) -- (270.31,37.48) -- (266.73,83.64) -- cycle ;
\draw [color={rgb, 255:red, 126; green, 211; blue, 33 }  ,draw opacity=1 ]   (197.49,55.04) -- (158.27,28) ;
\draw  [color={rgb, 255:red, 126; green, 211; blue, 33 }  ,draw opacity=1 ] (158.27,28) -- (146.67,85.67) -- (102.6,47.01) -- cycle ;
\draw  [color={rgb, 255:red, 74; green, 144; blue, 226 }  ,draw opacity=1 ] (419.05,74.24) -- (447.98,115.02) -- (407.2,143.95) -- (378.27,103.17) -- cycle ;
\draw  [color={rgb, 255:red, 255; green, 0; blue, 0 }  ,draw opacity=1 ] (468.43,42.27) -- (419.05,74.24) -- (415.94,15.7) -- cycle ;
\draw  [color={rgb, 255:red, 126; green, 211; blue, 33 }  ,draw opacity=1 ] (447.98,115.02) -- (490.73,97.25) -- (520.84,132.42) -- (496.7,171.93) -- (451.67,161.17) -- cycle ;
\draw [color={rgb, 255:red, 126; green, 211; blue, 33 }  ,draw opacity=1 ]   (490.73,97.25) -- (523.57,62.74) ;
\draw  [color={rgb, 255:red, 126; green, 211; blue, 33 }  ,draw opacity=1 ] (523.57,62.74) -- (468.43,42.27) -- (513.5,4.78) -- cycle ;

\draw (124,82.07) node [anchor=north west][inner sep=0.75pt]  [font=\small]  {$u_{1}$};
\draw (201,83.07) node [anchor=north west][inner sep=0.75pt]  [font=\small]  {$u_{2}$};
\draw (139,113.07) node [anchor=north west][inner sep=0.75pt]  [font=\small,color={rgb, 255:red, 208; green, 2; blue, 27 }  ,opacity=1 ]  {$h( S_{1})$};
\draw (95,185.67) node [anchor=north west][inner sep=0.75pt]   [align=left] {};
\draw (179,133.07) node [anchor=north west][inner sep=0.75pt]  [font=\small]  {$x$};
\draw (205,125.4) node [anchor=north west][inner sep=0.75pt]  [font=\small,color={rgb, 255:red, 208; green, 2; blue, 27 }  ,opacity=1 ]  {$\textcolor[rgb]{0.29,0.56,0.89}{h}\textcolor[rgb]{0.29,0.56,0.89}{(}\textcolor[rgb]{0.29,0.56,0.89}{S}\textcolor[rgb]{0.29,0.56,0.89}{_{2}}\textcolor[rgb]{0.29,0.56,0.89}{)}$};
\draw (401,63.07) node [anchor=north west][inner sep=0.75pt]  [font=\small]  {$x'$};
\draw (446,92.07) node [anchor=north west][inner sep=0.75pt]  [font=\small]  {$u'_{2}$};
\draw (456,20.07) node [anchor=north west][inner sep=0.75pt]  [font=\small]  {$u'_{1}$};
\draw (418,36.07) node [anchor=north west][inner sep=0.75pt]  [font=\small,color={rgb, 255:red, 208; green, 2; blue, 27 }  ,opacity=1 ]  {$h'( S_{1})$};
\draw (393,102.4) node [anchor=north west][inner sep=0.75pt]  [font=\small,color={rgb, 255:red, 208; green, 2; blue, 27 }  ,opacity=1 ]  {$\textcolor[rgb]{0.29,0.56,0.89}{h'}\textcolor[rgb]{0.29,0.56,0.89}{(}\textcolor[rgb]{0.29,0.56,0.89}{S}\textcolor[rgb]{0.29,0.56,0.89}{_{2}}\textcolor[rgb]{0.29,0.56,0.89}{)}$};

\end{tikzpicture}
\caption{A potential extension $\displaystyle \tilde{h}$ and $\displaystyle \widetilde{h'}$ of two embeddings $\displaystyle h,\ h':\ V( S_{1} \cup S_{2}) \ \rightarrow E$.}
\label{cyclespic2}
\end{figure}
We now weight the count of each pair of congruent embedding $h, h': V(S_1 \cup S_2) \to E$ by how many pairs of congruent embeddings $\tilde{h}, \tilde{h'}: V(\CC) \to E$ it extends to.  This is exactly

\begin{equation}
    \sum_{\delta'} F_{\delta'}(u_1, u_2) F_{\delta'}(u_1' u_2'),
\end{equation}
completing the proof of the proposition.

\end{proof}

We now prove a modified version of simplex unbalancing for cycles to reduce to the case where all but one simplex in $\mathcal{C}$ has dimension 1.  We have the following definition for convenience of notation, and the following Lemma:

\begin{definition}
    For a cycle of simplices $\CC$, denote $\ppc(\CC)$ as the infimum of $s \in \mathbb{R}$ such that $\stsumc \lesssim E^{2|V(\CC)|}q^{-c(\CC)}$ whenever $|E| \gtrsim q^s$.
\end{definition}

\begin{lemma} \label{cycleshelper}
    Given a cycle of simplices $\CC$, consider two distinct simplices $S_1$ and $S_2$ each of dimension at least 2, so each have at least one free vertex.  Consider the cycle of simplices $\CC'$ formed by deleting one free vertex of $S_1$ and adding one free vertex to $S_2$. Then we have that
    \begin{align*}
        \left( \frac{\Stab(\dim(S_1))}{\Stab(\dim(S_1'))} \cdot \frac{\Stab(\dim(S_2))}{\Stab(\dim(S_2'))} \right) q^{c(\CC ' )}\ =\ q^{c(\CC)}.
    \end{align*}
\end{lemma}
\begin{proof}
    This proof follows identically to Lemma \ref{step0lemma} as $c(\TT)$ and $c(\CC)$ are both defined to be the sum of $c(S)$ for all simplices $S$ in the structure.
\end{proof}

\begin{lemma}[Simplex Unbalancing for Cycles]\label{simplexunbalancingforcycles}
    Given a cycle of simplices $\CC$, consider two distinct simplices $S_1$ and $S_2$ each of dimension at least 2.  Consider the cycle of simplices $\CC_1$ formed by deleting all free vertices of $S_2$ and adding the number of free vertices deleted to $S_1$.  Similarly, consider the cycle of simplices $\CC_2$ formed by deleting all free vertices of $S_1$ and adding the number of free vertices deleted to $S_2$. Then $\ppc(\CC) \leq \max(\ppc(\CC_1), \ppc(\CC_2))$. 
\end{lemma}

\begin{proof} Throughout this proof, we denote the modification of the simplices $S_1$ and $S_2$ in $\CC_1$ as $S_1^+$ and $S_2^-$ respectively.  Similarly, we denote the modification of $S_1$ and $S_2$ in $\CC_2$ as $S_1^-$ and $S_2^+$. \bigskip

Suppose that $S_1$ and $S_2$ are nonadjacent, of dimensions $m$ and $n$ respectively.  Define the function $\psi(a_1, a_2): \mathbb{R}^2 \to \mathbb{R}$ as the expression $D_{\mathrm{nonadj}}(\CC, S_1, S_2)$, where we replace the exponents on $\lambda_{\theta}$ and $\lambda_{\psi}$ with $a_1$ and $a_2$ respectively.  The expression $D_{\mathrm{nonadj}}(\CC, S_1, S_2)$ can be written as $\psi(m-1, n-1)$, and by an application of Corollary \ref{repeatedhadamard}, we see that 
\begin{equation}
    \psi(m-1, n-1)\ \leq\ \psi(m+n-2, 0)^{s_1}  \psi(0, m+n-2)^{s_2}
\end{equation} 
for some positive $s_i$ such that $s_1 + s_2 = 1$.  Note that the only difference between $\psi(m+n-2, 0)$ and $D_{\mathrm{nonadj}}(\CC_1, S_1^+, S_2^-)$ is the stabilizer terms out front.  This is because in the process of applying Hadamard three-lines above, we changed the dimension of $S_1$ and $S_2$ by changing the power on the $\lambda_{\theta}$ terms, but didn't change the stabilizer terms accordingly.  Therefore, we have that

\begin{equation}
    \psi(k_1 + k_2, 0)\ =\ \left( \frac{\Stab(\dim(S_1)+\dim(S_2) - 1)}{\Stab(\dim(S_1))} \cdot \frac{\Stab(1)}{\Stab(\dim(S_2))} \right) D_{\mathrm{nonadj}}(\CC_1, S_1^+, S_2^-).
\end{equation}

By a repeated application of Lemma \ref{cycleshelper}, we have that
\begin{equation}
    \left( \frac{\Stab(\dim(S_1)+\dim(S_2) - 1)}{\Stab(\dim(S_1))} \cdot \frac{\Stab(1}{\Stab(\dim(S_2))} \right) \frac{1}{q^{c(\CC_1)}}\ =\ \frac{1}{q^{c(\CC)}}.
\end{equation}
After noting that $|V(\CC_1)| = |V(\CC)|$, we have that for $s > \ppc(\CC_1)$,
\begin{align}
    \psi(k_1 + k_2, 0)\ &\lesssim\ \left( \frac{\Stab(\dim(S_1)+k_2)}{\Stab(\dim(S_1))} \cdot \frac{\Stab(\dim(S_2)-k_2)}{\Stab(\dim(S_2))} \right) \left(\eoq{2|V(\CC_1)|}{c(\CC_1)}\right) \\ 
    &\lesssim \ \eoq{2|V(\CC)|}{c(\CC)}.
\end{align} 
By an analogous argument, we also have that for $s > \ppc(\CC_2)$, that $\psi(0, k_1 + k_2) \lesssim E^{2|V(\CC_2)|}q^{-c(\CC_2)}$. 
 Therefore, whenever $s > \max(\ppc(\CC_1),\ppc(\CC_2)) $, we may conclude that $\stsumc \approx \psi(k_1, k_2) \lesssim E^{2|V(\CC)|}q^{-c(\CC)}$. \bigskip

 The proof for adjacent simplices $S_1$ and $S_2$ follows similarly but with $D_{\text{adj}}$ instead. \bigskip
\end{proof}

Given any cycle of simplices $\mathcal{C}$ with at least three simplices, by applying simplex unbalancing repeatedly and noting that this operation preserves the sum of the dimensions of simplices in $\CC$, we arrive at Proposition \ref{goalof5.1}.  All that remains is to find an $s$ such that $E$ contains a positive proportion of congruence classes of $\CC'$ whenever $|E| \gtrsim q^s$.

\subsection{Technical obstructions}

For the rest of this section, we describe our attempts to bound the number of pairs of congruent embeddings of $\CC'$ in $E$, and the obstacles that prevent us from applying a generalization of the methods in Section 3 and \cite{preprint}. We assume that $N$, the dimension of the large simplex in $\CC'$, is larger than $d$ so that we don't have additional stabilizer terms in the sums that we're working with.  This is to add clarity to where technical obstructions arise, but the same issues occur with the stabilizer terms present.  \bigskip

We define $P_{t_1, \ldots, t_{k-1}}(u, u')$ as the number of paths $(x_1, \ldots, x_k)$ in $E$ of length $k-1$ such that $x_1 = u$, $x_2 = u'$, and $||x_{i+1} - x_i|| = t_i$ for all $1 \leq i \leq k-1$.  Let $\CC'$ be a cycle of simplices with one $N$-dimensional simplex and $k-1$ 1-dimensional simplices.  As in Section 4.3, to find an $s$ such that $E$ contains a positive proportion of congruence classes of $\CC'$ whenever $|E| \gtrsim q^s$, we bound the sum

\begin{equation}
    \stsumc\ \approx\ \sum_{t_1, \ldots, t_{k-1}} \sum_{\theta, w} \lambda_{\theta}^{N-1}(w) \sum_{\substack{u_1, u_1', u_2, u_2' \\ u_1 - \theta u_1' \\ = u2 - \theta u_2' \\ = w }} P_{t_1, \ldots, t_{k-1}}(u_1, u_2) P_{t_1, \ldots, t_{k-1}}(u_1', u_2')
\end{equation}
As in Section 4.3, we may assume that $t_1, \ldots t_k$ range over nonzero elements in $\fqd$ rather than all of $\fqd$.  We fix a congruence class of paths by fixing $t_1, \ldots t_{k-1}$, and drop the $t_i$s from the notation from $P$ for notational convenience.  Then, similar to the goal of Section 3, we want to show that

\begin{equation}
    \sum_{\theta, w} \lambda_{\theta}^{N-1}(w) \sum_{\substack{u, u', v, v' \\ u - \theta u' \\ = v - \theta v' \\ = w }} P(u, v) P(u', v') \ \lesssim\ {\eoq{2|V(\CC)|}{c(\CC)+k-1}}
\end{equation}
Following the work of Section 3 and of \cite{preprint}, we would like to define a function of the form
\begin{equation}
    \beta_{\theta}(w)\ =\ \sum_{\substack{u, u', v, v' \\ u - \theta u' = \\ v - \theta v' = w}} P(u, v) P(u', v') 
\end{equation}
We first calculate the zero Fourier coefficient of $\beta_{\theta}$ below.
\begin{align}
    \widehat{\beta}_{\theta}(0) & \ = \ \frac{1}{q^d} \sum_{\substack{u, v, u', u' \\ || u - v || = || u' - v' ||}} P(u, v)P(u', v') \\
    &= \  \frac{1}{q^d} \sum_{u, v} P(u, v) \sum_{\substack{u', v' \\ || u - v || = || u' - v' ||}} P(u', v') \\ 
    &\approx \ \frac{1}{q^d} \cdot \frac{|E|^k}{q^{k-1}} \cdot \frac{|E|^k}{q^k} \ = \ \frac{|E|^{2k}}{q^{d + 2k - 1}} \label{obstructionline2}
\end{align}

The first equality of \eqref{obstructionline2} follows from the fact that $E$ contains a statistically correct number of cycles of length $k$ for $|E| \gtrsim q^{(d+2)/2}$ by \cite{cycles} (with the technical condition that $k > 4$), and that $E$ contains a statistically correct number of paths of length $k-1$ for a lower threshold by \cite{longpaths}. \bigskip

We have that
\begin{equation}
    \sum_{\theta, w} \lambda_{\theta}^{N-1}(w) \sum_{\substack{u, u', v, v' \\ u - \theta u' \\ = v - \theta v' \\ = w }} P(u, v) P(u', v') \ = \ \sum_{\theta, w} \lambda_{\theta}^{N-1}(w) \beta(w) \label{futureworkneedtobound}
\end{equation}
This is a sum of a very similar form as the sum $\sum_{\theta, w} \lambda_{\theta}^n(w) \Gamma_{T_k, \theta}(w)$ in Section 3.  Therefore, we try to apply the method of proof of Lemma \ref{hammerindd} to bound this term.  However, this is not effective, since if we attempt to break up $\sum_{\theta, w}\lambda_{\theta}^{N-1}(w) \beta(w)$ in the same way as in Lemma \ref{hammerindd}, the term $\sum_{w, \theta} \lambda_{\theta}^{N-1}\widehat{\beta}_{\theta}(0)$ is too large by a factor of $q^{d-1}$. \bigskip

Another idea is to instead consider the function $\alpha_{\theta}: \fqd \times \fqd \to \mathbb{R}$ defined as 
\begin{equation}
    \alpha_{\theta}(w_1, w_2)\ =\ \sum_{\substack{u, u', v, v' \\ u - \theta u' = w_1 \\ v - \theta v' = w_2}} f(u, v) f(u', v') 
\end{equation}
We note that $\alpha_{\theta}(w, w) = \beta_{\theta}(w)$.  The benefit of $\alpha$ is that after taking the Fourier transform of $\alpha$ in both variables, $\widehat{\alpha}_{\theta}(0, 0)$ is smaller than $\widehat{\beta}_{\theta}(0)$.  More precisely, in a very similar manner as for $\Gamma_{T_k, \theta}$ in Section 3, we see that
\begin{equation}
    \widehat{\alpha}_{\theta}(m_1, m_2)\ =\ q^{2d} \widehat{P}(m_1, m_2) \overline{\widehat{P} (\theta^{-1}m_1, \theta^{-1}m_2)}
\end{equation}
We now see that 
\begin{equation}
    \widehat{\alpha}_{\theta}(0, 0)\ \approx\ \frac{|E|^{2k}}{q^{2d + 2k - 2}},
\end{equation} which is a factor of $q^{d-1}$ smaller than $\widehat{\beta}_{\theta}(0)$.  Therefore, we define $B_n$ as $\sum_{\theta,w}\lambda_\theta^n(w)\alpha_{\theta}(w, w)$ to break up our sum instead as 
\begin{align}
    B_n\ =\ \sum_{\theta,w}\lambda_\theta^n(w)\alpha_{\theta}(w, w)\ &\leq\ \left| \sum_{\theta,w}\lambda_\theta^{n-1}(w)\left(\lambda_\theta(w)-\frac{|E|^2}{q^d}\right)(\alpha_{\theta}(w, w)-\widehat{\alpha}_{\theta}(0, 0))\right|\nonumber\\
    &\hspace{0.5cm}+\  \sum_{w,\theta}\lambda_\theta^n(w)\widehat{\alpha}_{\theta}(0, 0)+\frac{|E|^2}{q^{d}}B_{n-1} \\
     &=:\ I + II+ III. \label{thethreenewterms}
\end{align}

As $\widehat{\alpha}_{\theta}(0, 0)$ is sufficiently small, the term $II$ is of the proper size.  Term $I$ can be bounded nontrivially by pulling absolute values inside, taking the supremum of $\lambda_{\theta}^{n-1}(w)$, and applying Cauchy Schwarz.  We handle the term of the form $\sum_w \left(\alpha_{\theta}(w, w) - \widehat{\alpha}_{\theta}(0,0)\right)^2$ by dominating the sum over $w$ by sums over $w_1, w_2$, and then running an analogous argument to \ref{Gammaminuszerofourierd}.  This gives us that
\begin{equation}
    \sum_{\theta, w} \left(\alpha_{\theta}(w, w) - 
    \widehat{\alpha}_{\theta}(0,0)\right)^2 \ \approx \ q^{2d}\left( \eoq{4k-4}{4k-4}\right).
\end{equation}
Plugging in bounds from \ref{lambdathetafactsd} gives us that $I$ is on the order of $E^{2|V(\CC)|}q^{-c(\CC)-k+1}$ for an $s$ slightly larger than that of an $N$ simplex, matching what we morally expect.  \bigskip

As we bound term $III$ by induction, it remains to bound the base case $B_1$, which we are unfortunately unable to bound effectively.  In the spirit of how we bound $A_1$ in Lemma \ref{hammerindd}, one may easily generalize Lemma \ref{aimlemma39} to functions $\fqd \times \fqd \to \mathbb{R}$ by modifying the proof seen in \cite{groupactions}.  However, our bound for $B_1$ is still too big by a factor of $q^{d-1}$. If we reduce the base case to $B_0$, the $L_1$ norm of $\alpha$ is too large by a factor of $q^{d-1}$.  Finally, if we break up $B_1$ slightly differently in hopes to extract cancellation from the negative term, we are still off by the same factor.  \bigskip

The recurring issue is that we need for one of $\alpha$ or $\beta$, both its $L_1$ norm and zero Fourier coefficient to be small.  While $\widehat{\alpha}_{\theta}(0, 0)$ is smaller than $\widehat{\beta}_{\theta}(0)$, the $L_1$ norm of $\alpha_{\theta}$ is much larger than $\beta_{\theta}$.  We have this difference because we get an extra factor of $q^d$ in the denominator by taking the Fourier transform of both coordinates for $\alpha_{\theta}$.  What we believe needs to happen is for the negative contribution of term $I$ (without the absolute values) to cancel with part of term $III$ of \ref{thethreenewterms}.  Therefore, we need to somehow extract cancellation between terms $I$ and $III$, potentially involving results on number-theoretic sums.  We also propose a potential workaround in Section 6.2.



\section{Future work}

\subsection{Edge gluing}

We end this paper by discussing a natural generalization of the work in the past two sections, considering structures of simplices where two simplices may share an edge or some higher dimensional face rather than a vertex.  The Hadamard three-lines framework established in Section 4 can be applied to such a situation, allowing us to reduce to a simple class of simplex structures that encourages the use of a combination of group action and inductive techniques.  To illustrate this, we work through a simple example. \bigskip

\begin{figure}[hbt!]
\centering

\tikzset{every picture/.style={line width=1pt}} 

\begin{tikzpicture}[x=0.75pt,y=0.75pt,yscale=-1,xscale=1]

\draw  [color={rgb, 255:red, 208; green, 2; blue, 27 }  ,draw opacity=1 ] (521,100.17) -- (473.9,165) -- (397.69,140.23) -- (397.69,60.1) -- (473.9,35.34) -- cycle ;
\draw [color={rgb, 255:red, 208; green, 2; blue, 27 }  ,draw opacity=1 ]   (397.69,60.1) -- (473.9,165) ;
\draw [color={rgb, 255:red, 208; green, 2; blue, 27 }  ,draw opacity=1 ]   (397.69,140.23) -- (521,100.17) ;
\draw [color={rgb, 255:red, 208; green, 2; blue, 27 }  ,draw opacity=1 ]   (397.69,60.1) -- (521,100.17) ;
\draw [color={rgb, 255:red, 208; green, 2; blue, 27 }  ,draw opacity=1 ]   (473.9,35.34) -- (473.9,165) ;
\draw [color={rgb, 255:red, 208; green, 2; blue, 27 }  ,draw opacity=1 ]   (397.69,140.23) -- (473.9,35.34) ;
\draw [color={rgb, 255:red, 74; green, 144; blue, 226 }  ,draw opacity=1 ]   (473.9,35.34) -- (542.11,40.33) ;
\draw [color={rgb, 255:red, 74; green, 144; blue, 226 }  ,draw opacity=1 ]   (521,100.17) -- (542.11,40.33) ;
\draw [color={rgb, 255:red, 0; green, 0; blue, 0 }  ,draw opacity=1 ]   (473.9,35.34) -- (521,100.17) ;
\draw    (163.11,48.33) -- (163.11,146.33) ;
\draw [color={rgb, 255:red, 74; green, 144; blue, 226 }  ,draw opacity=1 ]   (163.11,48.33) -- (251.11,116.56) ;
\draw [color={rgb, 255:red, 74; green, 144; blue, 226 }  ,draw opacity=1 ]   (163.11,146.33) -- (235.75,61.98) ;
\draw [color={rgb, 255:red, 74; green, 144; blue, 226 }  ,draw opacity=1 ]   (163.11,48.33) -- (235.75,61.98) ;
\draw [color={rgb, 255:red, 74; green, 144; blue, 226 }  ,draw opacity=1 ]   (235.75,61.98) -- (251.11,116.56) ;
\draw [color={rgb, 255:red, 74; green, 144; blue, 226 }  ,draw opacity=1 ]   (163.11,146.33) -- (251.11,116.56) ;
\draw [color={rgb, 255:red, 208; green, 2; blue, 27 }  ,draw opacity=1 ]   (163.11,146.33) -- (74.21,79.75) ;
\draw [color={rgb, 255:red, 208; green, 2; blue, 27 }  ,draw opacity=1 ]   (163.11,48.33) -- (91.61,134.02) ;
\draw [color={rgb, 255:red, 208; green, 2; blue, 27 }  ,draw opacity=1 ]   (163.11,146.33) -- (90.3,134.04) ;
\draw [color={rgb, 255:red, 208; green, 2; blue, 27 }  ,draw opacity=1 ]   (90.3,134.04) -- (74.21,79.75) ;
\draw [color={rgb, 255:red, 208; green, 2; blue, 27 }  ,draw opacity=1 ]   (161.81,48.35) -- (74.21,79.75) ;

\draw (58,188) node [anchor=north west][inner sep=0.75pt]   [align=left] {{\small $\displaystyle G$, two 3-simplices sharing an edge}};
\draw (442,190) node [anchor=north west][inner sep=0.75pt]   [align=left] {$ $};
\draw (341,190) node [anchor=north west][inner sep=0.75pt]   [align=left] {{\small $\displaystyle H$, a 4-simplex and a 1-simplex sharing an edge}};
\end{tikzpicture}
\caption{An example of Hadamard three-lines simplification for edge-gluing.}
\label{edgegluingfig}
\end{figure}

Consider the graph $G$ formed by two 3-simplices $S_1, S_2$ sharing an edge, as shown in Figure \ref{edgegluingfig}.  We would like to find an $s$ such that for $E \subset \mathbb{F}_q^3$, $E$ contains a positive proportion of congruence classes of $G$ whenever $|E| \gtrsim q^s$.  Note that there are on the order of $q^{11}$ congruence classes of $G$ in $\mathbb{F}_q^d$.  Denoting $\mathbb{D}$ as the set of congruence classes of $G$, and $\nu_G(\delta)$ as the number of embeddings of $G$ in $E$ in the congruence class $\delta$, by the Cauchy Schwarz setup as presented in Section 4 it is sufficient to find an $s$ such that for $|E| \gtrsim q^s$,
\begin{equation}
    \sum_{\delta \in \mathbb{D}} \nu_G^2(\delta) \ \lesssim\ {\eoq{12}{11}}
\end{equation}
For two embeddings $h_1, h_2: V(G) \to E$ of $G$, $h_1$ and $h_2$ are congruent if and only if there exist $\theta, \phi \in \oq{3}$ and $w_1, w_2 \in \mathbb{F}_q^3$ such that $h_1(v) = \theta h_2(v) + w_1$ for all $v \in S_1$ and $h_1(v) = \phi h_2(v) + w_2$ for all $v \in S_2$.  Therefore, we see by a similar argument as in Section 4.1 that 

\begin{equation}
     \sum_{\delta \in \mathbb{D}} \nu_G^2(\delta) \ \approx\ \sum_{\theta, \phi \in \oq{d}} \sum_{\substack{x, y \in \mathbb{F}_q^d \\ x - \theta x' = y - \theta y' \\ x - \phi x' = y - \phi y'} }\lambda_{\theta}^2( x - \theta x') \lambda_{\phi}^2(x - \phi x')
\end{equation}
Note that in general stabilizer terms are introduced (as in Section 4.1), but in this example the stabilizer terms are trivial.  Applying Hadamard three-lines in the style of simplex unbalancing gives us that
\begin{equation}
     \sum_{\delta \in \mathbb{D}} \nu_G^2(\delta) \ \lesssim\ \sum_{\theta, \phi \in \oq{d}} \sum_{\substack{x, y \in \mathbb{F}_q^d \\ x - \theta x' = y - \theta y' \\ x - \phi x' = y - \phi y'} }\lambda_{\theta}^3( x - \theta x') \lambda_{\phi}^1(x - \phi x'), \label{futureworkgasum}
\end{equation}
where the right hand side is on the order of $\sum_{\delta \in \mathbb{D}} \nu_H^2(\delta)$, where $H$ is the graph shown in Figure \ref{edgegluingfig}.  Therefore, to find an $s$ for $G$, it suffices to find an $s$ such that for $|E| \gtrsim q^s$,
\begin{equation}
    \sum_{\delta \in \mathbb{D}} \nu_H^2(\delta) \ \lesssim\ {\eoq{12}{11}}.
\end{equation}


In more complicated simplex structures where all simplices are connected either by vertices or edges, an inductive Hadamard three-lines approach similar to Section 4 can reduce to the case where $\TT$ is a large simplex $S$ where at a single vertex, standard (1-simplex) trees are attached, and at a single edge, tree structures formed by gluing 2-links (shown in Figure \ref{twolink}) are attached.  \bigskip

\begin{figure}[hbt!]
\centering

\tikzset{every picture/.style={line width=1pt}} 

\begin{tikzpicture}[x=0.75pt,y=0.75pt,yscale=-1,xscale=1]

\draw    (176,44) -- (202.53,93.33) ;
\draw    (176,44) -- (234.11,44) ;
\draw    (234.11,44) -- (202.11,93) ;
\draw    (202.53,93.33) -- (260.11,93.33) ;
\draw    (234.11,44) -- (260.11,93.33) ;
\draw    (234.11,44) -- (292.22,44) ;
\draw    (260.11,93.33) -- (317.69,93.33) ;
\draw    (317.69,93.33) -- (375.27,93.33) ;
\draw    (292.22,44) -- (350.33,44) ;
\draw    (350.33,44) -- (408.44,44) ;
\draw    (375.27,93.33) -- (432.84,93.33) ;
\draw    (292.22,44) -- (260.11,93.33) ;
\draw    (349.8,44) -- (317.69,93.33) ;
\draw    (407.38,44) -- (375.27,93.33) ;
\draw    (291.69,44) -- (317.69,93.33) ;
\draw    (350.33,44) -- (376.33,93.33) ;
\draw    (408.44,44) -- (432.84,93.33) ;

\end{tikzpicture}
\caption{A 2-link of length 8}
\label{twolink}
\end{figure}
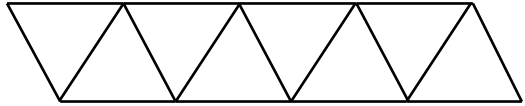

Therefore, 2-links and trees of 2-links form a natural analogue to paths and trees in the edge-gluing setting.  Their inductive nature suggests that there exists an $s$ such that $E \subset \mathbb{F}_q^d$, $|E| \gtrsim q^s$ implies that $E$ contains a statistically correct number of 2-links of arbitrary length in any congruence class, similar to chains and trees of simplices as in the work of \cite{longpaths} and \cite{small2023}.  Considering simplex structures where simplices can share up to $n$-dimensional faces, a similar Hadamard three-lines approach can be used to reduce to a large simplex with $k$-links attached (defined analogously) for $1 \leq k \leq n$.  \bigskip

In attempting to bound $\sum_{\delta \in \mathbb{D}} \nu_H^2(\delta)$ in a similar manner to Section 5.2, we run into similar technical obstructions.  We propose a different approach in the following subsection.

\subsection{Thresholds}

Intuitively, it is clear that if $E$ contains a positive proportion of congruence classes in $\fqd$ of a graph $G$ whenever $E \gtrsim q^s$, then for $E \gtrsim q^{s+\varepsilon}$, $E$ should contain "almost all" of the congruence classes of $G$ in $\mathbb{F}_q^d$.  However, the rate at which this threshold phenomenon acts, and the precise definition of "almost all," is of importance as it may provide a different approach to finding such an $s$ for graphs $G$ and $H$ sharing a vertex, edge, or more complicated subgraph that avoids the obstacles presented in Sections 5 and 6.  \bigskip

Suppose we have two arbitrary graphs $G$ and $H$, and we pick an $s$ such that for $E \subset \mathbb{F}_q^d, |E| \gtrsim q^s$, $E$ contains a positive proportion of congruence classes of $G$ in $\mathbb{F}_q^d$, and $E$ contains a statistically correct number of embeddings of every congruence class of $H$ in $\mathbb{F}_q^d$.  This is the setup that we have in Section 3, where $G$ is a simplex and $H$ is a $k$-weak simplex tree.  Being able to quantify how many congruence classes of $G$ that $E$ contains when $|E| \gtrsim q^{s+\varepsilon}$ could open the door to showing that $E$ contains a positive proportion of congruence classes of $G$ attached to $H$ at a vertex/edge/some subgraph via a probabilistic argument and concentration inequalities.  An approach such as this would avoid the need to work with the explicit sums  corresponding to $G$ and $H$ attached at some subgraph, bypassing the obstacles found in Sections 5 and 6 that occur while trying to separate out the individual contributions from $G$ and $H$ in these sums.

\section{Acknowledgements}

We thank Prof. Brian McDonald for several helpful conversations. We also thank Jacques Hadamard for his great result as we celebrate his 159 birthday this year.\bigskip

This paper was written as part of the SMALL REU Program 2024.  We appreciate the support of Williams College, as well as funding from NSF grant DMS-2241623, Duke University, Emmanuel College Cambridge, the Finnerty Fund, Princeton University, University of Michigan, and Williams College. A.I. was supported in part by the NSF under grant no. 2154232. A.I. wishes to thank the Isaac Newton Institute (INI) in Cambridge, Great Britain, for their support and hospitality. This work was completed during the INI program, entitled "Multivariate approximation, discretization, and sampling recovery."

\printbibliography
\end{document}